\documentclass[11pt]{article}
\usepackage{amssymb, amsmath, enumerate, multirow, adjustbox, hhline, amsthm, tikz, geometry, pdflscape, rotating, hyperref}
\newtheorem{propn}{Proposition}
\newtheorem{defn}{Definition}
\newtheorem{corol}{Corollary}
\newtheorem{lemma}{Lemma}
\newtheorem{theorem}{Theorem}
\title{An Analysis of the Diaconis-Holmes-Neal Markov Chain Sampler Under Generalized Unimodal Underlying Probabilities}
\author{Martin V. Hildebrand\thanks{Department of Mathematics and Statistics, University at Albany, SUNY, Albany, NY 12222 USA E-mail: {\tt mhildebrand@albany.edu}} \and Christopher J. Lange\thanks{Department of Mathematics and Statistics, University at Albany, SUNY, Albany, NY 12222 USA E-mail: {\tt clange@albany.edu}}}
\date{\today}
\begin{document}

    \maketitle

    \begin{abstract}
        \noindent Upon the introduction of the Metropolis algorithm \cite{Metropolis1953}, the question of how many steps in the Markov chain were needed to achieve convergence to stationarity became apparent.  The convergence was rather slow, i.e. for a process on $n$ states the number of steps needed to achieve convergence to stationarity was found to be on the order of $n^2$ if the underlying distribution is uniform. \cite{DHN2000}\\

        \noindent The obvious problem with Metropolis et. al \cite{Metropolis1953} is that the Markov chain is reversible. In other words, for any state $j$ we can move from $j$ to $j + 1$ and back to $j$ in two steps.  To correct for this, Diaconis, Holmes, and Neal \cite{DHN2000} improved Metropolis et. al \cite{Metropolis1953} by introducing a non-reversible Markov chain.  The Diaconis-Holmes-Neal sampler, as it is known, is a Markov chain on two copies of $n$ states, a $+1$ copy and a $-1$ copy.  \\

        \noindent Applications of the Diaconis-Holmes-Neal sampler include Markov chain sampling and situations in statistical physics, among others \cite{DHN2000}.  However, an answer to the question of how many steps are needed to achieve convergence to stationarity was required.  In \cite{Hildebrand2002}, Hildebrand showed that if the underlying probabilities are log-concave then the sampler achieves convergence to stationarity in at least a constant multiple of $n$ steps.  Nonetheless, the question of whether a similar convergence exists when the underlying probabilities are instead unimodal was posed in Hildebrand \cite{Hildebrand2002}. While Lange \cite{Lange2025a} answered the question in the three simplest cases - the simple case, the function of $n$ case, and the asymmetric function of $n$ case - and Lange \cite{Lange2025b} answered the question in the general symmetric unimodal case, the general unimodal case is left to this paper.
    \end{abstract}

    \section{Introduction}

    \noindent Metropolis et. al \cite{Metropolis1953} introduced a reversible Markov chain on $n$ states with transition probabilities from state to state equal to 50\%.  Given underlying probabilities $\pi(x)$, for $1 \leq x \leq n$, the algorithm works as follows:
    \begin{enumerate}
	    \item Suppose the Markov chain is at state $j$.  With probability equal to 50\%, let $k = j + 1$.  Otherwise, let $k = j - 1$.
	    \item For $1 \leq k \leq n$ and with probability $p$, where $p$ is defined by
        \begin{equation}
            p = \min\left\{1, \frac{\pi(k)}{\pi(j)}\right\},
        \end{equation}
        \noindent the Markov chain advances to state $k$. Otherwise, the algorithm stays at state $j$.
    \end{enumerate}
    \noindent A diagram of the Metropolis algorithm is Figure~\ref{metropolisdiagram}.
    \begin{figure}[ht!]
	   \centering
	   \begin{tikzpicture}[scale = 0.8]
		  \foreach \x in {1,5,12,16}{
			 \draw[ultra thick] (\x, 0) circle [radius = 1]; 
		  }
		  \node(dots) at (8.5,0) {\text{\textbf{\dots}}}; 
		
		  \draw[<->, blue, ultra thick] (2.2,0) -- (3.8,0);
		  \draw[<->, purple, ultra thick] (6.2,0) -- (8,0);
		  \draw[<->, violet, ultra thick] (8.9,0) -- (10.8,0);
		  \draw[<->, magenta, ultra thick] (13.2,0) -- (14.8,0);
		
		  \node(left) at (-0.4,0.7){};
		  \node(right) at (17.4,-0.7){};
		  \draw[<-, red, ultra thick] (left)arc(45:315:1);
		  \draw[<-, green, ultra thick] (right)arc(45:315:-1);
	   \end{tikzpicture}
	   \caption{A diagram of the Metropolis algorithm.}
       \label{metropolisdiagram}
    \end{figure}
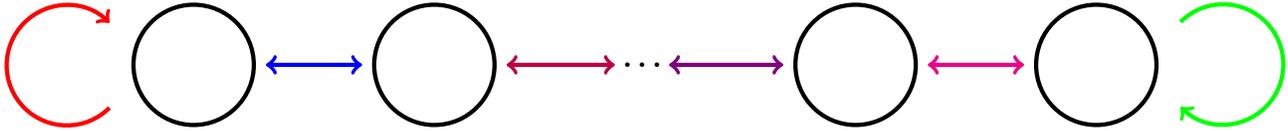
    \noindent Diaconis, Holmes, and Neal \cite{DHN2000} used the Central Limit Theorem to prove the following about the number of steps needed for the algorithm to converge to stationarity. 
    \begin{theorem}
	   For $\pi(j) = n^{-1}$ with $1 \leq j \leq n$, the number of steps needed for the Metropolis algorithm to converge to stationarity is on the order of $n^2$. \label{theorem1}
    \end{theorem}
    \noindent This theorem presents a question to be answered. What is meant by ``convergence to stationarity"? First, a definition.
    \begin{defn}
        Suppose $P(s)$ and $Q(s)$ are probabilities on a finite space $S$. The \underline{variation} \underline{distance} between $P$ and $Q$ is given by
        \begin{equation}
            ||P - Q|| = \frac{1}{2}\sum_{s \in S}\Biggl[\Bigl|P(s) - Q(s)\Bigr|\Biggr].
        \end{equation} \label{variationdistance}
    \end{defn}
    \noindent With this, we say that convergence to stationarity in any number of steps occurs when the variation distance between the probability of the algorithm advancing from any state $y$ in the specified number of steps and the stationary distribution of the algorithm is less than $\delta$ for some constant $\delta > 0$. From this theorem, we see a problem with the Metropolis algorithm when $n$ is large.\\

    \noindent To improve the convergence, Diaconis, Holmes, and Neal \cite{DHN2000} devised their algorithm. It utilizes two copies of each of the $n$ states, a $+1$ copy and a $-1$ copy, and allows for the Markov chain to change from the $+1$ copy to the $-1$ copy and vice versa.  The algorithm works as follows:
    \begin{enumerate}
        \item Let $\epsilon = \pm 1$, $1 \leq j \leq n$, and $0 < \theta < 1$. Suppose the Markov chain is at state $(\epsilon, j)$ and suppose the underlying probabilities are given by $\pi(j)$ with $\pi(0) = \pi(n + 1) = 0$.
        \item The Markov chain will advance to $(-\epsilon, j + \epsilon)$ with probability $\alpha$, where
        \begin{equation}
            \alpha = \min\left\{1, \frac{\pi(j + \epsilon)}{\pi(j)}\right\}
        \end{equation}
        \noindent and stay put otherwise.
        \item With probability $1 - \theta$, the Markov chain will change the sign of the first coordinate and stay put otherwise.
    \end{enumerate}
    \noindent \textbf{Note.}  One step of the Markov chain, as defined by Diaconis, Holmes, and Neal \cite{DHN2000}, is completed after Steps 2 and 3 are performed. Also, in this paper, $\theta = n^{-1}$.\\

    \noindent The Diaconis-Holmes-Neal sampler, as the algorithm is more commonly known, has four possible moves.  They are:
    \begin{enumerate}
	   \item \textbf{The shift.}  When a shift occurs, the Diaconis-Holmes-Neal sampler moves from $(+1, j)$ to $(+1, j + 1)$ or from $(-1, j)$ to $(-1, j - 1)$ with probability equal to
	   \begin{align}
		  p_s &= (1 - \theta)\min\left\{1, \frac{\pi(j + 1)}{\pi(j)}\right\}\\
		  &= \min\left\{1 - \theta, \frac{(1 - \theta)\pi(j + 1)}{\pi(j)}\right\}
	   \end{align}
	   \noindent if $\epsilon = +1$ and
	   \begin{align}
		  p_s &= (1 - \theta)\min\left\{1, \frac{\pi(j - 1)}{\pi(j)}\right\}\\
		  &= \min\left\{1 - \theta, \frac{(1 - \theta)\pi(j - 1)}{\pi(j)}\right\}
	   \end{align}
	   \noindent if $\epsilon = -1$.
	   \item \textbf{The jump.}  When a jump occurs, the Diaconis-Holmes-Neal sampler moves from $(+1, j)$ to $(-1, j)$ or vice versa with probability equal to
	   \begin{align}
		  p_j &= (1 - \theta)\left(1 - \min\left\{1, \frac{\pi(j + 1)}{\pi(j)}\right\}\right)\\
		  &= (1 - \theta)\max\left\{0, 1 - \frac{\pi(j + 1)}{\pi(j)}\right\}\\
		  &= (1 - \theta)\max\left\{0, \frac{\pi(j) - \pi(j + 1)}{\pi(j)}\right\}\\
		  &= \max\left\{0, \frac{(1 - \theta)(\pi(j) - \pi(j + 1))}{\pi(j)}\right\}
	   \end{align}
	   \noindent if $\epsilon = +1$ and
	   \begin{align}
		  p_j &= (1 - \theta)\left(1 - \min\left\{1, \frac{\pi(j - 1)}{\pi(j)}\right\}\right)\\
		  &= (1 - \theta)\max\left\{0, 1 - \frac{\pi(j - 1)}{\pi(j)}\right\}\\
		  &= (1 - \theta)\max\left\{0, \frac{\pi(j) - \pi(j - 1)}{\pi(j)}\right\}\\
		  &= \max\left\{0, \frac{(1 - \theta)(\pi(j) - \pi(j - 1))}{\pi(j)}\right\}
	   \end{align}
	   \noindent if $\epsilon = -1$.
	   \item \textbf{The flip.}  When a flip occurs, the Diaconis-Holmes-Neal sampler moves from $(+1, j)$ to $(-1, j + 1)$ or from $(-1, j)$ to $(+1, j - 1)$ with probability equal to
	   \begin{align}
		  p_f &= \theta\min\left\{1, \frac{\pi(j + 1)}{\pi(j)}\right\}\\
		  &= \min\left\{\theta, \frac{\theta\pi(j + 1)}{\pi(j)}\right\}
	   \end{align}
	   \noindent when $\epsilon = +1$ and
	   \begin{align}
		  p_f &= \theta\min\left\{1, \frac{\pi(j - 1)}{\pi(j)}\right\}\\
		  &= \min\left\{\theta, \frac{\theta\pi(j - 1)}{\pi(j)}\right\}
	   \end{align}
	   \noindent when $\epsilon = -1$.
	   \item \textbf{The stationary move.}  When a stationary move occurs, the Diaconis-Holmes-Neal sampler does not change states.  This occurs with probability equal to
	   \begin{align}
		  p_{f'} &= \theta\left(1 - \min\left\{1, \frac{\pi(j + 1)}{\pi(j)}\right\}\right)\\
		  &= \theta\max\left\{0, 1 - \frac{\pi(j + 1)}{\pi(j)}\right\}\\
		  &= \theta\max\left\{0, \frac{\pi(j) - \pi(j + 1)}{\pi(j)}\right\}\\
		  &= \max\left\{0, \frac{\theta(\pi(j) - \pi(j + 1))}{\pi(j)}\right\}
	   \end{align}
	   \noindent if $\epsilon = +1$ and
	   \begin{align}
		  p_{f'} &= \theta\left(1 - \min\left\{1, \frac{\pi(j - 1)}{\pi(j)}\right\}\right)\\
		  &= \theta\max\left\{0, 1 - \frac{\pi(j - 1)}{\pi(j)}\right\}\\
		  &= \theta\max\left\{0, \frac{\pi(j) - \pi(j - 1)}{\pi(j)}\right\}\\
		  &= \max\left\{0, \frac{\theta(\pi(j) - \pi(j - 1))}{\pi(j)}\right\}
	   \end{align}
	   \noindent if $\epsilon = -1$.
    \end{enumerate}
    \noindent \textbf{Note.}  
    \begin{enumerate}
        \item Diaconis, Holmes, and Neal \cite{DHN2000}, Hildebrand \cite{Hildebrand2002}, and Hildebrand \cite{Hildebrand2004} do not distinguish between flips and stationary moves.  Furthermore, they only distinguish between any move with probability on the order of $\theta$ and any move with probability on the order of $1 - \theta$.  The distinction among the four possible moves is made here for clarity.
        \item Hildebrand \cite{Hildebrand2004} gives a diagram of the possible moves listed above. 
        \item Hildebrand \cite{Hildebrand2002} introduces the concept of a basepoint - a state $\tilde{j}$ such that $\pi(\tilde{j}) \geq \pi(j)$ for every state $j$.  This is not discussed in Diaconis, Holmes, and Neal. \cite{DHN2000}
    \end{enumerate}
    \noindent A diagram of the Diaconis-Holmes-Neal sampler is shown in Figure~\ref{dhnsampler}.
    \begin{figure}[ht!]
	   \centering
	   \begin{tikzpicture}[scale = 0.6]
		  \foreach \x in {1,4,10,13,16,22,25}{
		  \draw[magenta, ultra thick] (\x, 7) circle [radius = 1];
		  \draw[magenta, ultra thick] (\x, 3) circle [radius = 1];
		  }
		  \node(topdots1) at (7, 7){\text{\textbf{\dots}}};
		  \node(bottomdots1) at (7, 3){\text{\textbf{\dots}}};
		  \node(topdots2) at (19, 7){\text{\textbf{\dots}}};
		  \node(bottomdots2) at (19, 3){\text{\textbf{\dots}}};
		  \foreach \x in {2,5,8,11,14,17,20,23}{
		  \draw[->, green, ultra thick](\x, 7) -- (\x + 1, 7);
		  }
		  \foreach \x in {3,6,9,12,15,18,21,24}{
		  \draw[->, green, ultra thick](\x, 3) -- (\x - 1, 3);
		  }
		  \foreach \x in {1,4,7,10,13,16,19,22,25}{
		  \draw[<->, blue, ultra thick](\x, 4) -- (\x, 6);
		  }
		  \foreach \x in {1,4,7,10,13,16,19,22}{
		  \draw[<->, red, ultra thick](\x + .7071068, 6.2928932) -- (\x + 2.2928932, 3.7071068);
		  }
		  \node(top1) at (1.5, 8.5){};
		  \node(bottom1) at (0.5, 1.5){};
		  \node(top2) at (4.5, 8.5){};
		  \node(bottom2) at (3.5, 1.5){};
		  \node(top3) at (7.5, 8.5){};
		  \node(bottom3) at (6.5, 1.5){};
		  \node(top4) at (10.5, 8.5){};
		  \node(bottom4) at (9.5, 1.5){};
		  \node(top5) at (13.5, 8.5){};
		  \node(bottom5) at (12.5, 1.5){};
		  \node(top6) at (16.5, 8.5){};
		  \node(bottom6) at (15.5, 1.5){};
		  \node(top7) at (19.5, 8.5){};
		  \node(bottom7) at (18.5, 1.5){};
		  \node(top8) at (22.5, 8.5){};
		  \node(bottom8) at (21.5, 1.5){};
		  \node(top9) at (25.5, 8.5){};
		  \node(bottom9) at (24.5, 1.5){};
		  \draw [<-, black, ultra thick] (top1)arc(-60:240:1);
		  \draw [<-, black, ultra thick] (bottom1)arc(-60:240:-1);
		  \draw [<-, black, ultra thick] (top2)arc(-60:240:1);
		  \draw [<-, black, ultra thick] (bottom2)arc(-60:240:-1);
		  \draw [<-, black, ultra thick] (top3)arc(-60:240:1);
		  \draw [<-, black, ultra thick] (bottom3)arc(-60:240:-1);
		  \draw [<-, black, ultra thick] (top4)arc(-60:240:1);
		  \draw [<-, black, ultra thick] (bottom4)arc(-60:240:-1);
		  \draw [<-, black, ultra thick] (top5)arc(-60:240:1);
		  \draw [<-, black, ultra thick] (bottom5)arc(-60:240:-1);
		  \draw [<-, black, ultra thick] (top6)arc(-60:240:1);
		  \draw [<-, black, ultra thick] (bottom6)arc(-60:240:-1);
		  \draw [<-, black, ultra thick] (top7)arc(-60:240:1);
		  \draw [<-, black, ultra thick] (bottom7)arc(-60:240:-1);
		  \draw [<-, black, ultra thick] (top8)arc(-60:240:1);
		  \draw [<-, black, ultra thick] (bottom8)arc(-60:240:-1);
		  \draw [<-, black, ultra thick] (top9)arc(-60:240:1);
		  \draw [<-, black, ultra thick] (bottom9)arc(-60:240:-1);
	   \end{tikzpicture}
	   \caption{A diagram of the Diaconis-Holmes-Neal sampler.}
       \label{dhnsampler}
    \end{figure}
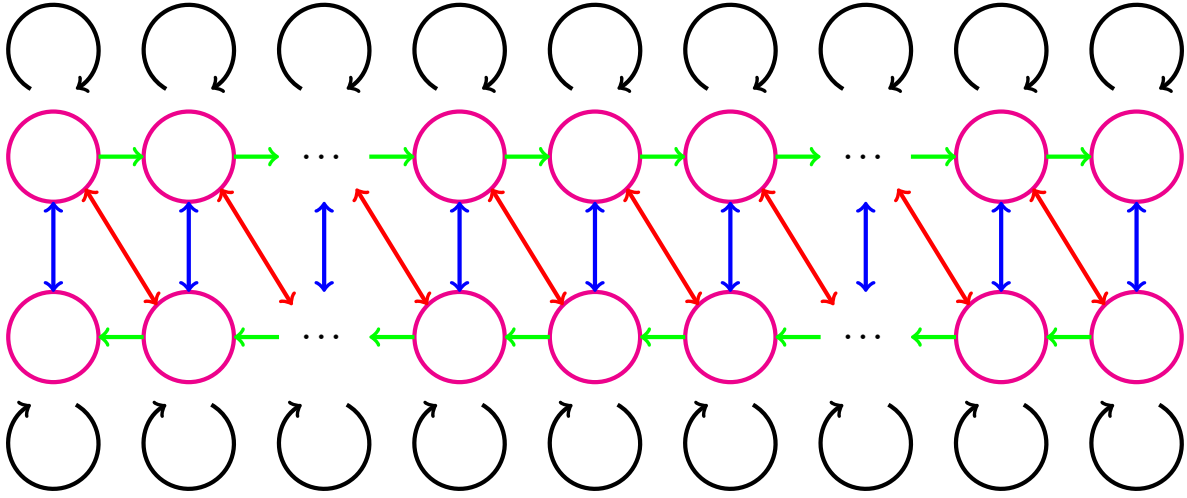
    \noindent From here, the question becomes how quickly the Diaconis-Holmes-Neal sampler converges to stationarity. Hildebrand \cite{Hildebrand2002} proved that $z$ is some constant $c$ multiple of $n$ steps if the underlying probabilities are log-concave, i.e. $(\pi(j))^2\ge \pi(j-1)\pi(j+1)$ for $j=2, 3,..., n-1$. (We assume $\pi(j)>0$ for $j=1, 2,..., n$.) The main theorem from this work is stated below.
    \begin{theorem}
	   Suppose 
	   \begin{enumerate}[(a)]
		  \item $\theta = n^{-1}$,
		  \item $\tilde{X}$ is the Diaconis-Holmes-Neal sampler with log-concave non-zero probabilities $\pi$,
		  \item $\tilde{\pi}$ is the stationary probability of the Diaconis-Holmes-Neal sampler (i.e. $\tilde{\pi}((\epsilon,j))=(1/2)\pi(j)$),
		  \item $P_z$ is the probability of $\tilde{X}$ moving from state $y$ to state $j'$ in exactly $z$ steps, and
		  \item $\delta > 0$.
	   \end{enumerate}
	   \noindent Then there exists a constant $c > 0$ such that
	   \begin{enumerate} [(1)]
		  \item $c$ does not depend on $\pi$, and
		  \item $||P_z - \tilde{\pi}|| < \delta$ for $z \geq cn$.
	   \end{enumerate} \label{Hildebrand2002Thm}
    \end{theorem}
    \noindent One critical result used to prove the theorem above is Theorem 16.2.4 Meyn and Tweedie \cite{MeynTweedie2009} and is stated as the following theorem.
    \begin{theorem}
	   If a chain $\Phi$ satisfies
	   \begin{equation}
		  P^m(x, A) \geq \nu_m(A)
	   \end{equation}
	   \noindent for all $x \in X$ and $A \in {\cal B}(X)$, then
	   \begin{equation}
		  ||P^n(x, \centerdot) - \pi|| \leq \rho^{\lfloor n/m\rfloor} \label{1.1}
	   \end{equation}
	   \noindent where $\rho = 1 - \nu_m(X)$. \label{MeynTweedie2009Thm}
    \end{theorem}
    \noindent \textbf{Note.} The statement of the theorem in Meyn and Tweedie \cite{MeynTweedie2009} multiplies the upper bound on the variation distance by two. This is due to the fact that Meyn and Tweedie reference the supremum definition of variation distance. The definition of variation distance specified previously allows us to omit this factor of two.\\

    \noindent To prove Theorem~\ref{Hildebrand2002Thm}, Hildebrand \cite{Hildebrand2002} first proved that a path of length $z = cn$ exists from state $y$ to state $j'$ when the underlying probabilities $\pi$ are log-concave.  Hildebrand \cite{Hildebrand2002} defines a \underline{path} as a sequence of transitions in the Diaconis-Holmes-Neal sampler from $y$ to $j'$, and the paths constructed each have one flip or stationary move.  From here, a lower bound on the occurrence probability of the constructed path is found. This lower bound is then used to find a lower bound on the occurrence probability of a path starting at state $y$ and ending in some subset of states of the Markov chain. (Usually this is found by multiplying by some constant or a multiple of $n$.) Finally, Theorem~\ref{MeynTweedie2009Thm} is used with the second lower bound that is found and a limit is taken as $n \to \infty$. \cite{Hildebrand2002}\\

    \noindent Hildebrand \cite{Hildebrand2002} posed the question of whether a similar convergence holds when $\pi$ is unimodal.   Stated differently, if $\pi$ is any probability with a single maximum value, does convergence to stationarity hold in a constant multiple of $n$ steps? Lange \cite{Lange2025a} proved that such a convergence holds in the three simplest cases of $\pi$ being unimodal, namely the simple case, the (symmetric) function of $n$ case, and the asymmetric function of $n$ case, and the ideas were extended to the general symmetric unimodal case in Lange \cite{Lange2025b}. The convergence in each case is being proved by:
    \begin{enumerate}
	   \item Proving that paths of length $x = cn$ exist from $y$ to $j'$,
        \item Finding a lower bound on the occurrence probability of these paths occurring,
        \item Finding a lower bound on the probability that a path of length $x = cn$ starts at $y$ and ends in some subset of states of the Markov chain,
        \item Applying Theorem~\ref{MeynTweedie2009Thm} to the result found in Step 3, and
	   \item Taking the limit of the result as $n \to \infty$.
    \end{enumerate}
    \noindent \textbf{Note.}  While the possibility of having more flips or stationary moves exists, due to the nature of the underlying probabilities, most paths built in Lange \cite{Lange2025a} only have one or two flips or stationary moves. Hildebrand \cite{Hildebrand2002} only considered paths with one. \\

    \noindent In this paper, we answer the question posed by Hildebrand \cite{Hildebrand2002} in the general unimodal case. This is discussed in Section~\ref{generalcase}. Some closing remarks and future research questions are stated in Section~\ref{conclusion}.

    \section{The General Unimodal Case} \label{generalcase}

    \subsection{Introducing the General Unimodal Case}

    \noindent Let $n$ be odd and suppose $\{m_j(n)\}_{j = 1}^n$ is a sequence of real-valued functions on $n$ such that both of the following apply:
    \begin{enumerate}
        \item $m_j(n) > 0$ for every $1 \leq j \leq n$, and
        \item The sequence is unimodal, i.e. for some value $j^*$, $m_1(n)\le m_2(n)\le ...\le m_{j^*}(n)\ge m_{j^*+1}(n)\ge ...\ge m_n(n)$.
    \end{enumerate}
    \noindent Here, the normalizing constant is
    \begin{equation}
        z' = \sum_{j = 1}^n\Biggl[m_j(n)\Biggr],
    \end{equation}
    \noindent and we can define the underlying probabilities by
    \begin{equation}
        \pi(j) = \frac{m_j(n)}{z'}. \label{3.1}
    \end{equation}
    \noindent We need not split the definition into multiple cases because we are not assuming that the sequence is symmetric. As the assumption of symmetry from the general symmetric case is relaxed here, we produce a diagram of the Diaconis-Holmes-Neal sampler and the permissible moves when the underlying probabilities are given by \eqref{3.1}. We also provide the probabilities of each permissible move in the general unimodal case.
    \begin{table}[ht!]
        \centering
        \scalebox{0.75}{\begin{tabular}{| c c | c c | c c |}\hline
            \multirow{ 3}{*}{} & \multirow{ 3}{*}{} & \multirow{ 3}{*}{} & \multirow{ 3}{*}{} & \multirow{ 3}{*}{} & \multirow{ 3}{*}{}\\
            \scalebox{0.4}{\begin{tikzpicture} \draw[->, green, ultra thick] (0,0) -- (2,0); \end{tikzpicture}} & $p = \dfrac{n - 1}{n}$ & \scalebox{0.4}{\begin{tikzpicture} \draw[->, cyan!60, ultra thick] (0,0) -- (2,0); \end{tikzpicture}} & $p = \dfrac{m_{j + 1}(n)}{m_j(n)}\left(\dfrac{n - 1}{n}\right)$ & \scalebox{0.4}{\begin{tikzpicture} \draw[->, cyan!20, ultra thick] (0,0) -- (2,0); \end{tikzpicture}} & $p = \dfrac{m_{j - 1}(n)}{m_j(n)}\left(\dfrac{n - 1}{n}\right)$\\
             & & & & & \\ \hline
            \multirow{ 3}{*}{} & \multirow{ 3}{*}{} & \multirow{ 3}{*}{} & \multirow{ 3}{*}{} & \multirow{ 3}{*}{} & \multirow{ 3}{*}{}\\
            \scalebox{0.4}{\begin{tikzpicture} \draw[->, purple, ultra thick] (0,0) -- (0,2); \end{tikzpicture}} & $p = \dfrac{n - 1}{n}$ & \scalebox{0.4}{\begin{tikzpicture} \draw[->, blue!60, ultra thick] (0,0) -- (0,2); \end{tikzpicture}} & $p = \dfrac{n - 1}{n}\left(1 - \dfrac{m_{j + 1}(n)}{m_j(n)}\right)$ & \scalebox{0.4}{\begin{tikzpicture} \draw[->, blue!20, ultra thick] (0,0) -- (0,2); \end{tikzpicture}} & $p = \dfrac{n - 1}{n}\left(1 - \dfrac{m_{j - 1}(n)}{m_j(n)}\right)$\\
             & & & & & \\ \hline
            \multirow{ 3}{*}{} & \multirow{ 3}{*}{} & \multirow{ 3}{*}{} & \multirow{ 3}{*}{} & \multirow{ 3}{*}{} & \multirow{ 3}{*}{}\\
            \scalebox{0.4}{\begin{tikzpicture} \draw[->, red, ultra thick] (0,0) -- (2,2); \end{tikzpicture}} & $p = \dfrac{1}{n}$ & \scalebox{0.4}{\begin{tikzpicture} \draw[->, red!60, ultra thick] (0,0) -- (2,2); \end{tikzpicture}} & $p = \dfrac{m_{j + 1}(n)}{nm_j(n)}$ & \scalebox{0.4}{\begin{tikzpicture} \draw[->, red!20, ultra thick] (0,0) -- (2,2); \end{tikzpicture}} & $p = \dfrac{m_{j - 1}(n)}{nm_j(n)}$\\
             & & & & & \\ \hline
            \multirow{ 3}{*}{} & \multirow{ 3}{*}{} & \multirow{ 3}{*}{} & \multirow{ 3}{*}{} & \multirow{ 3}{*}{} & \multirow{ 3}{*}{}\\
            \scalebox{0.4}{\begin{tikzpicture} \node(n1) at (39.5,13.5){}; \draw [<-, black, ultra thick] (n1)arc(-60:240:1); \end{tikzpicture}} & $p = \dfrac{1}{n}$ & \scalebox{0.4}{\begin{tikzpicture} \node(n3) at (51.5,13.5){}; \draw [<-, brown!60, ultra thick] (n3)arc(-60:240:1); \end{tikzpicture}} & $p = \dfrac{1}{n}\left(1 - \dfrac{m_{j + 1}(n)}{m_j(n)}\right)$ & \scalebox{0.4}{\begin{tikzpicture} \node(n3) at (51.5,13.5){}; \draw [<-, brown!20, ultra thick] (n3)arc(-60:240:1); \end{tikzpicture}} & $p = \dfrac{1}{n}\left(1 - \dfrac{m_{j - 1}(n)}{m_j(n)}\right)$\\
             & & & & & \\ \hline
        \end{tabular}}
        \caption{The probabilities associated with each possible move in the general unimodal case.}
    \end{table}

    \pagebreak

    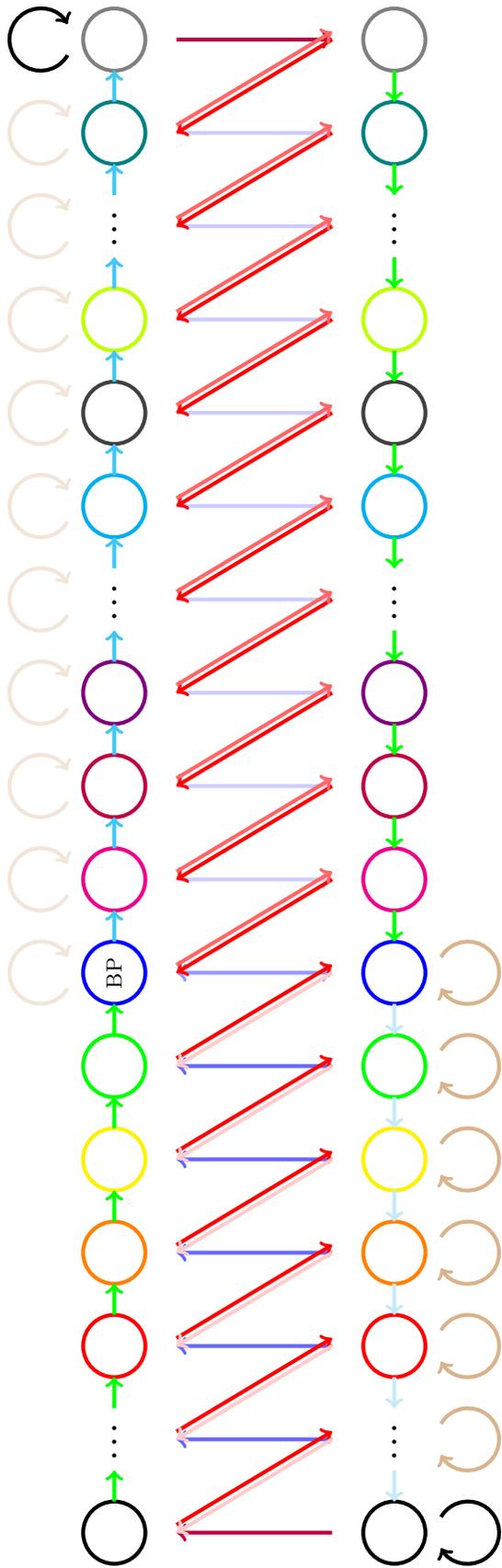
\begin{figure}[ht!]
        \begin{sideways}
            \centering
            \begin{tikzpicture}[scale = 0.45]
                \draw[ultra thick] (3,12) circle [radius = 1];
                \draw[red, ultra thick] (9,12) circle [radius = 1];
                \draw[orange, ultra thick] (12,12) circle [radius = 1];
                \draw[yellow, ultra thick] (15,12) circle [radius = 1];
                \draw[green, ultra thick] (18,12) circle [radius = 1];
                \draw[blue, ultra thick] (21,12) circle [radius = 1];
                \draw[magenta, ultra thick] (24,12) circle [radius = 1];
                \draw[purple, ultra thick] (27,12) circle [radius = 1];
                \draw[violet, ultra thick] (30,12) circle [radius = 1];
                \draw[cyan, ultra thick] (36,12) circle [radius = 1];
                \draw[darkgray, ultra thick] (39,12) circle [radius = 1];
                \draw[lime, ultra thick] (42,12) circle [radius = 1];
                \draw[teal, ultra thick] (48,12) circle [radius = 1];
                \draw[gray, ultra thick] (51,12) circle [radius = 1];

                \draw[ultra thick] (3,3) circle [radius = 1];
                \draw[red, ultra thick] (9,3) circle [radius = 1];
                \draw[orange, ultra thick] (12,3) circle [radius = 1];
                \draw[yellow, ultra thick] (15,3) circle [radius = 1];
                \draw[green, ultra thick] (18,3) circle [radius = 1];
                \draw[blue, ultra thick] (21,3) circle [radius = 1];
                \draw[magenta, ultra thick] (24,3) circle [radius = 1];
                \draw[purple, ultra thick] (27,3) circle [radius = 1];
                \draw[violet, ultra thick] (30,3) circle [radius = 1];
                \draw[cyan, ultra thick] (36,3) circle [radius = 1];
                \draw[darkgray, ultra thick] (39,3) circle [radius = 1];
                \draw[lime, ultra thick] (42,3) circle [radius = 1];
                \draw[teal, ultra thick] (48,3) circle [radius = 1];
                \draw[gray, ultra thick] (51,3) circle [radius = 1];

                \node(basepoint) at (21,12) {\footnotesize BP};
		
			    \foreach \x in {2,11,15}{
			    \node(top) at (3 * \x, 12){\text{\textbf{\dots}}};
			    \node(bottom) at (3 * \x, 3){\text{\textbf{\dots}}};
			    }
		
			    \foreach \x in {1,2,3,4,5,6}{
		        \draw[->, green, ultra thick](3 * \x + 1, 12) -- (3 * \x + 2, 12);
			    }
			    \foreach \x in {7,8,9,10,11,12,13,14,15,16}{
			    \draw[<-, green, ultra thick](3 * \x + 1, 3) -- (3 * \x + 2, 3);
			    }
		
			    \foreach \x in {7,8,9,10,11,12,13,14,15,16}{
                \draw[->, cyan!60, ultra thick] (3 * \x + 1, 12) -- (3 * \x + 2, 12);
                }
                \foreach \x in {1,2,3,4,5,6}{
                \draw[<-, cyan!20, ultra thick] (3 * \x + 1, 3) -- (3 * \x + 2, 3);
                }

                \draw[->, purple, ultra thick] (3,5) -- (3,10);
		        \draw[<-, purple, ultra thick] (51,5) -- (51,10);
			
                \foreach \x in {2,3,4,5,6}{
                \draw[->, blue!60, ultra thick] (3 * \x,5) -- (3 * \x,10);
                }
                \foreach \x in {8,9,10,11,12,13,14,15,16}{
                \draw[<-, blue!20, ultra thick] (3 * \x,5) -- (3 * \x,10);
                }
                \draw[<->, blue!40, ultra thick] (21,5) -- (21,10);
			
			    \foreach \x in {1,2,3,4,5,6}{
		        \draw[->, red, ultra thick] (3 * \x + .25, 10) -- (3 * \x + 3.25, 5);}
			    \foreach \x in {7,8,9,10,11,12,13,14,15,16}{
		        \draw[<-, red, ultra thick] (3 * \x, 10) -- (3 * \x + 3, 5);}
		
		        \foreach \x in {7,8,9,10,11,12,13,14,15,16}{
		        \draw[->, red!60, ultra thick] (3 * \x + .25, 10) -- (3 * \x + 3.25, 5);}
		        \foreach \x in {1,2,3,4,5,6}{
		        \draw[<-, red!20, ultra thick] (3 * \x, 10) -- (3 * \x + 3, 5);}
			
		        \node(n2) at (51.5,13.5){};
	            \node(n3) at (2.5,1.5){};
		        \draw [<-, black, ultra thick] (n2)arc(-60:240:1);
		        \draw [<-, black, ultra thick] (n3)arc(-60:240:-1);

                \foreach \x in {2,3,4,5,6,7}{
                \node(bottom) at (3 * \x - 0.5, 1.5){};
                \draw[<-, brown!60, ultra thick] (bottom)arc(-60:240:-1);
                }
                \foreach \x in {7,8,9,10,11,12,13,14,15,16}{
                \node(top) at (3 * \x + 0.5,13.5){};
                \draw [<-, brown!20, ultra thick] (top)arc(-60:240:1);
                }
		    \end{tikzpicture}
        \end{sideways}
        \caption{A diagram of the general unimodal case. The basepoint is labeled.}
    \end{figure}

    \pagebreak

    \noindent Before we proceed to prove the main result, we need to define the state that will serve as the basepoint.   The basepoint will be $\tilde{j} = (+1, j^*)$.

    \subsection{Proving the Main Result}

    \noindent The objective, here, is to prove the following result.
    \begin{theorem}
        Suppose
        \begin{enumerate}
            \item The Diaconis-Holmes-Neal sampler $\Phi$ on $n$ states, with $n$ odd, has underlying probabilities given by Equation \eqref{3.1},
            \item $\{m_j(n)\}_{j = 1}^n$ is a sequence of real-valued functions such that
            \begin{enumerate}
                \item $m_j(n) > 0$ for every $1 \leq j \leq n$, and
                \item The sequence is unimodal,
            \end{enumerate}
            \item $z'$ is the normalizing constant,
            \item $c$ and $c'$ are separate constants,
            \item $x = cn$, and
            \item $z \geq c'n$.
        \end{enumerate}
        \noindent Then, for some $\hat{c} > 0$ that may depend on $\pi(j)$,
        \begin{equation}
           \lim_{n \to \infty}\Biggl[||P^z(y,A) - \tilde{\pi}||\Biggr] \lesssim \left(1 - \frac{\hat{c}}{e^c}\right)^{\lfloor c'/c \rfloor} \label{3.2}
        \end{equation} \label{mainresult}
    \end{theorem}
    \noindent \textbf{Note.} Without loss of generality, we will assume that $j^* \leq (n + 1)/2$. Similar arguments will prove our results in the case where $j^* > (n + 1)/2$.\\

    \noindent Before we outline a plan to prove Theorem~\ref{mainresult}, we first provide the following definition.
    \begin{defn}
        A path from $y$ to $j'$ is said to be \underline{ideal} if its length is exactly $cn$ steps, and it is said to have an \underline{overshoot} if its length is longer than $cn$. \label{ideal}
    \end{defn}

    \noindent The steps we follow to prove Theorem~\ref{mainresult} do not change from Lange \cite{Lange2025a}. Step 1, the path construction step, is shown in Section~\ref{sectionstep1}, followed by the lower bound calculations (Steps 2 and 3) in Section~\ref{sectionstep2} and Section~\ref{sectionstep3}. Finally, Steps 4 and 5 are given in Section~\ref{sectionstep4and5}.

    \subsubsection{Step 1} \label{sectionstep1}

    \noindent We start with Step 1, which is to prove the following.
    \begin{lemma}
        Suppose the Diaconis-Holmes-Neal sampler $\Phi$ on $n$ states, with $n$ odd, has underlying probabilities given by Equation \eqref{3.1}, where $\{m_j(n)\}_{j = 1}^n$ is a unimodal sequence of real-valued functions on $n$ that take on positive values, and where $z'$ is the normalizing constant. Then, there exists a constant $c$, which does not depend on $n$, and an ideal path from state $y$ to state $j'$. \label{step1}
    \end{lemma}
    \noindent \textit{Proof.} While we give the entire proof for clarity, it is very similar to the proof of \textbf{Lemma 9} from Lange \cite{Lange2025a}. Here, we have the addendum of splitting the Diaconis-Holmes-Neal sampler into four separate sectors. Namely:
    \begin{enumerate}
        \item A state $w$ is said to be in the \underline{upper-right sector} if $w = (+1,j)$ where $j > j^*$.
        \item A state $w$ is said to be in the \underline{lower-right sector} if $w = (-1,j)$ where $j > j^*$.
        \item A state $w$ is said to be in the \underline{lower-left sector} if $w = (-1,j)$ where $j \leq j^*$.
        \item A state $w$ is said to be in the \underline{upper-left sector} if $w = (+1,j)$ where $j \leq j^*$.
    \end{enumerate}
    \noindent With the sectors defined, we can proceed with the path construction in the general unimodal case. As we did with the asymmetric function of $n$ case, we split a path from $y$ to $j'$ into the standard three components. We construct the first component ($y$ to the basepoint) first, then the third (the basepoint to $j'$), followed by the second (intermediate returns to the basepoint).\\

    \noindent \underline{\textit{First component.}} There are four possibilities here:
    \begin{enumerate}
        \item Suppose $y$ is in the upper-left sector. Then the first component consists of a sub-path entirely made of shift moves. The number of steps this sub-path has is between 0 and $j^* - 1$.
        \item Suppose $y$ is in the lower-left sector. Then the first component consists of a sub-path made of any number of shift moves along with a jump. The number of steps this sub-path has is between 1 and $2j^* - 1$.
        \item Suppose $y$ is in the lower-right sector. The sub-path that the first component consists of has the same construction as the sub-path in the previous case. However, the number of steps this sub-path has is between 2 and $n + j^* - 1$.
        \item Suppose $y$ is in the upper-right sector. Then the first component consists of a sub-path made of several shift moves along with two jumps. The number of steps this sub-path has is between 3 and $2n - 1$.
    \end{enumerate}
    \noindent \underline{\textit{Third component.}} This is constructed very similarly to the first component. However, the sub-path construction is flipped here. More specifically:
    \begin{enumerate}
        \item If $j'$ is in the upper-right sector, the construction is similar to the first scenario of the construction of the first component. The number of steps this sub-path has is between 1 and $n - j^*$.
        \item If $j'$ is in the lower-right sector, the construction is similar to the second scenario of the construction of the first component. The number of steps this sub-path has is between 2 and $2(n - j^*)$.
        \item If $j'$ is in the lower-left sector, the construction is similar to the third scenario of the construction of the first component. The number of steps this sub-path has is between 1 and $2n - j^*$.
        \item If $j'$ is in the upper-left sector, the construction is similar to the fourth scenario of the construction of the first component except if $j'$ is the basepoint, no steps are needed. The number of steps this sub-path has is between 0 and $2n - 1$.
    \end{enumerate}
    \noindent \underline{\textit{Second component.}} As was the case with the general symmetric case, we consider the following two possibilities:
    \begin{enumerate}
        \item Suppose an intermediate return does not have a flip or a stationary move. Then, the number of steps this intermediate return will have is between 2 and $2n$.
        \item Suppose an intermediate return has exactly one flip or stationary move. Then, the number of steps this intermediate return will have is between 3 and $4n - 2j^* - 1$.
    \end{enumerate}
    \noindent From here, we ``glue" the components together as Lange \cite{Lange2025a} did in the symmetric function of $n$ case. However, this may lead to the path possibly having an overshoot. A diagram of what could happen when ``gluing" is shown in Figure~\ref{glueovershoot}.\\
    
    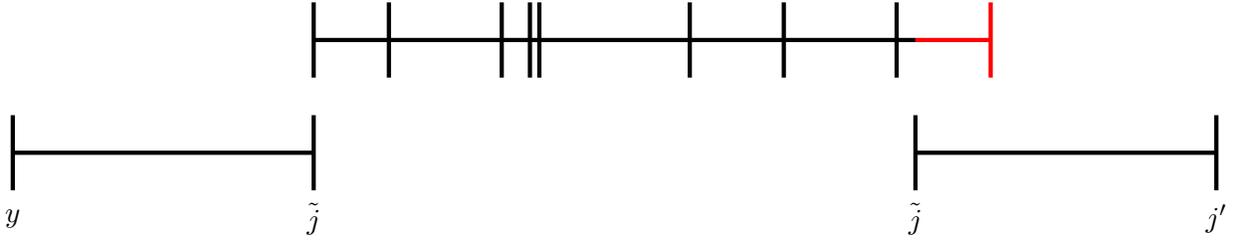
\begin{figure}[ht!]
        \centering
        \begin{tikzpicture}
            \draw[ultra thick] (0,-0.5) -- (0,0.5);
            \node(n1) at (0,-0.875) {$y$};
            \draw[ultra thick] (4,-0.5) -- (4,0.5);
            \node(n2) at (4,-0.875) {$\tilde{j}$};
            \draw[ultra thick] (0,0) -- (4,0);
            \draw[ultra thick] (12,-0.5) -- (12,0.5);
            \node(n3) at (12,-0.875) {$\tilde{j}$};
            \draw[ultra thick] (16,-0.5) -- (16,0.5);
            \node(n4) at (16,-0.875) {$j'$};
            \draw[ultra thick] (12,0) -- (16,0);
            \draw[ultra thick] (4,1) -- (4,2);
            \draw[ultra thick] (5,1) -- (5,2);
            \draw[ultra thick] (6.5,1) -- (6.5,2);
            \draw[ultra thick] (6.875,1) -- (6.875,2);
            \draw[ultra thick] (7,1) -- (7,2);
            \draw[ultra thick] (9,1) -- (9,2);
            \draw[ultra thick] (10.25,1) -- (10.25,2);
            \draw[ultra thick] (11.75,1) -- (11.75,2);
            \draw[red, ultra thick] (13,1) -- (13,2);
            \draw[ultra thick] (4,1.5) -- (12,1.5);
            \draw[red, ultra thick] (12,1.5) -- (13,1.5);
        \end{tikzpicture}
        \caption{A diagram of an attempt to ``glue" the first component and the third component with a second component containing no flips or stationary moves. The overshoot that is likely to occur is shown in \textcolor{red}{red}.}
        \label{glueovershoot}
    \end{figure}
    \noindent If an overshoot exists, we will need to shorten one or more intermediate returns by the amount of the overshoot, or the difference between the length of the path and $cn$. How does this ``shortening" work? We will either...
    \begin{itemize}
        \item ...shorten one or more intermediate returns with a flip or stationary move, or 
        \item ...include a flip or stationary move in one or more intermediate returns without one. 
    \end{itemize}

    \newpage
    
    \noindent In most cases, we will include a flip or stationary move in an intermediate return without one. The flips or stationary moves are placed so that the intermediate returns take fewer steps. A diagram of shortening an intermediate return that already has a flip or stationary move is shown in Figure~\ref{shortwithflip}.\\
    
    \begin{figure}[ht!]
        \centering
        \begin{tikzpicture}
            \draw[ - , green, ultra thick] (-4,3) -- (2,3);
            \draw[ - , ultra thick] (2,3) -- (4,3);
            \draw[ - , ultra thick] (4,3) -- (5,0);
            \draw[ - , green, ultra thick] (-4,0) -- (3,0);
            \draw[ - , ultra thick] (3,0) -- (5,0);
            \draw[ - , green, ultra thick] (-4,0) -- (-4,3);
            \draw[ - , green, ultra thick] (2,3) -- (3,0);
            \draw[red, fill = red] (0,3) circle [radius = 0.2];
            \draw[ <- , green, ultra thick] (3,1.5) -- (4,1.5);
        \end{tikzpicture}
        \caption{A representation of shortening an intermediate return with a flip or stationary move. The \textcolor{red}{red} dot represents the basepoint.}
        \label{shortwithflip}
    \end{figure}
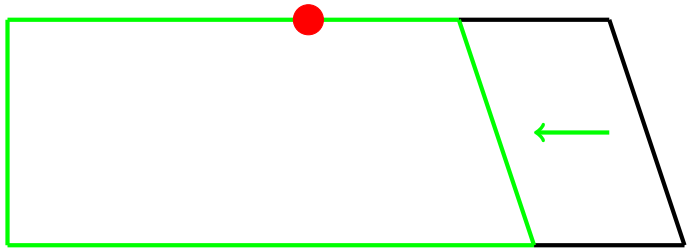
    \noindent We next show how an intermediate return that does not have a flip or stationary move can be shortened by including one. A diagram of this is shown in Figure~\ref{shortwithoutflip}.\\
    
    \begin{figure}[ht!]
        \centering
        \begin{tikzpicture}
            \draw[ - , green, ultra thick] (-4,3) -- (2,3);
            \draw[ - , red, ultra thick] (2,3) -- (5,3);
            \draw[ - , red, ultra thick] (5,3) -- (5,0);
            \draw[ - , green, ultra thick] (-4,0) -- (3,0);
            \draw[ - , red, ultra thick] (3,0) -- (5,0);
            \draw[ - , green, ultra thick] (-4,0) -- (-4,3);
            \draw[ - , green, ultra thick] (2,3) -- (3,0);
            \draw[red, fill = red] (0,3) circle [radius = 0.2];
            \draw[ <- , green, ultra thick] (3,1.5) -- (4.5,1.5);
        \end{tikzpicture}
        \caption{A representation of shorting an intermediate return by including a flip or stationary move. The \textcolor{red}{red} dot represents the basepoint.}
        \label{shortwithoutflip}
    \end{figure}
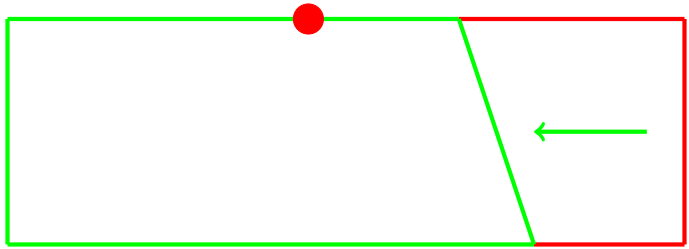
    \noindent These acts of shortening will only serve to make the constructed path at least as likely to occur conditioned on where the flips do and do not occur. This can be described more precisely as follows. Let $L_1$ be the event the path taken by the intermediate return has no flips or stationary moves, transitions from the $+1$ copy to the $-1$ copy on step number $k_0$,
    and transitions from the $-1$ copy to the $+1$ copy at $(-1,k')$,
    and let $L_2$ be the event the path taken by the path has a flip on step number $k$ (where $k<k_0$), has no other transition from the $+1$ copy to the $-1$ copy, 
    has no other flip or stationary move, and transitions from the $-1$ copy to the $+1$ copy at $(-1,k')$.
    Let $A_1$ be the event that occurs when there are no flips or stationary moves in the intermediate return, and let $A_2$ be the event there is precisely one flip or stationary move in the intermediate return and this flip or stationary move occurs on step number $k$. Then $P(L_1|A_1)\le P(L_2|A_2)$.\\
    
    \noindent This all being said, there are a few things we need to be aware of:
    \begin{enumerate}
        \item The case where a flip or stationary move is contained in the last intermediate return can be ignored, as the probability of this occurring is bounded above by a positive constant smaller than one.
        \item If the amount of the overshoot is larger than $3n - 1$, we must shorten at least two intermediate returns to obtain a path of length $x = cn$.
        \item If the amount of the overshoot is an even integer, an even number of intermediate returns (usually two) will need to be shortened.
        \item Any number of intermediate returns may be shortened by including flips or stationary moves, or by moving flips or stationary moves. Nowhere in the proof of this lemma is there a requirement to shorten as few intermediate returns as necessary.
        \item Using a flip or stationary move to shorten an intermediate return without one allows us to shorten the intermediate return by up to $2d^{**} - 1$ steps, where $1 \leq d_1 < j^* < d_2 \leq n$ are such that all of the following apply:
        \begin{enumerate}
            \item $d_2 > j^*$ is as large as possible such that
            \begin{equation}
                \sum_{j^* \leq t_2 \leq d_2}\left[(2(t_2 - j^*) + 1)\left(1 - \frac{m_{t_2}(n)}{m_{j^*}(n)}\right)\right] \leq \frac{1}{2}\sum_{j^* \leq t_2 \leq n}\left[(2(t_2 - j^*) + 1)\left(1 - \frac{m_{t_2}(n)}{m_{j^*}(n)}\right)\right]. \label{3.3}
            \end{equation}
            \noindent If this fails for every $d_2 > j^*$, we set $d_2 = j^* + 1$.
            \item $d_1 < j^*$ is as small as possible such that
            \begin{equation}
                \sum_{d_1 \leq t_1 \leq j^*}\left[(2(j^* - t_1) + 1)\left(1 - \frac{m_{t_1}(n)}{m_{j^*}(n)}\right)\right] \leq \frac{1}{2}\sum_{1 \leq t_1 \leq j^*}\left[(2(j^* - t_1) + 1)\left(1 - \frac{m_{t_1}(n)}{m_{j^*}(n)}\right)\right].\label{3.4}
            \end{equation}
            \noindent Likewise, if this fails for every $d_1 < j^*$, we set $d_1 = j^* - 1$.
            \item $d^{**}$ is such that
            \begin{equation}\label{d2stars}
                d^{**}=\max\{d_2-j^{*},j^{*}-d_1\}
            \end{equation}
        \end{enumerate}
        \item With positive probability,
        \begin{enumerate}
            \item Most intermediate returns will not have ``many" steps, and
            \item The ``longer" intermediate returns will take the ``majority" of the steps in the second component.
        \end{enumerate}
    \end{enumerate}
    \noindent What is meant by ``many steps", ``longer", and ``majority of the steps" is completely determined by the underlying probability distribution $\pi(j)$. Because of conclusions 6(a) and 6(b), it is more likely than not that including a flip or stationary move anywhere will lead to the flip or stationary move appearing in the longer intermediate returns.\\

    \noindent Shortening an appropriate number of intermediate returns by including a flip or stationary move in apprpriate locations in each, we obtain a path that has exactly $cn$ steps. Thus, the path is ideal. We are done.\qed\\

    \noindent \textbf{Note.}
    \begin{enumerate}
        \item Equation \eqref{3.3} gives $d_2$ such that the weighted sum of shifting $d_2 - j^*$ steps, jumping, and then shifting another $d_2 - j^*$ steps is less than or equal to 50\% of the weighted sum of shifting $(n - 1)/2$ steps, jumping, and then shifting another $(n - 1)/2$ steps. The same logic can be applied to Equation \eqref{3.4}, $d_1$, and $j^* - d_1$. Observe
        \begin{equation}
            d^{**} \geq \frac{d_2 - d_1}{2}.
        \end{equation}
        \noindent Furthermore, because of conclusions 6(a) and 6(b), we can conclude that, with probability bounded below by the positive constant $c''$, it is equally likely to shorten an intermediate return without a flip or stationary move by any odd number of steps up to and including $2d^{**}- 1$ when including one. Here, $c''$ is related to $d_1$ and $d_2$, which by Equation \eqref{3.3} and Equation \eqref{3.4} implies that $c''$ is related to the underlying probabilities. Why is any of this true? Well, consider the diagram in Figure~\ref{intermediatereturns}.
        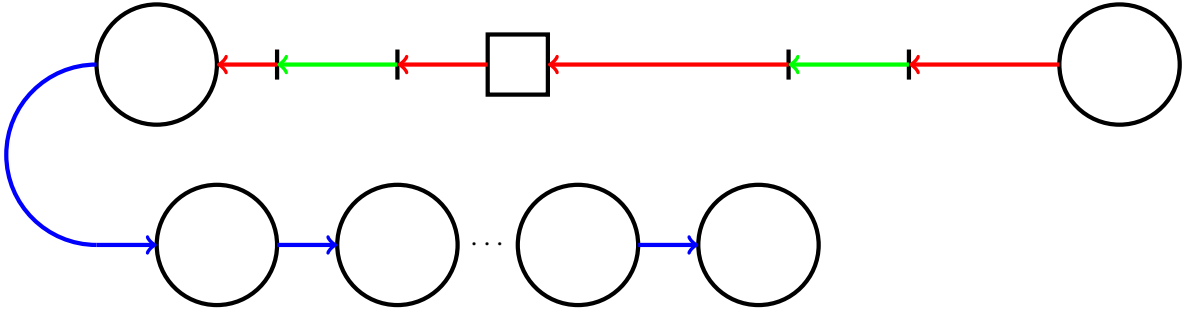
\begin{figure}[ht!]
            \centering
            \begin{tikzpicture}[scale = 0.8]
                \draw[ultra thick] (1,3) circle [radius = 1]; 
                \draw[ultra thick] (17,3) circle [radius = 1]; 
                \draw[ultra thick] (2,0) circle [radius = 1]; 
                \draw[ultra thick] (5,0) circle [radius = 1]; 
                \draw[ultra thick] (8,0) circle [radius = 1]; 
                \draw[ultra thick] (11,0) circle [radius = 1]; 

                \draw[ultra thick] (6.5,2.5) rectangle (7.5,3.5); 

                \node(dots) at (6.53125,0) {$\cdots$}; 
            
                \draw[ultra thick] (13.5,2.75) -- (13.5,3.25); 
                \draw[ultra thick] (11.5,2.75) -- (11.5,3.25); 
                \draw[ultra thick] (5,2.75) -- (5,3.25); 
                \draw[ultra thick] (3,2.75) -- (3,3.25); 

                \draw[->, red, ultra thick] (16,3) -- (13.5,3); 
                \draw[->, red, ultra thick] (11.5,3) -- (7.5,3); 
                \draw[->, red, ultra thick] (6.5,3) -- (5,3); 
                \draw[->, red, ultra thick] (3,3) -- (2,3); 
                \draw[->, blue, ultra thick] (0,0) -- (1,0); 
                \draw[->, blue, ultra thick] (3,0) -- (4,0); 
                \draw[->, blue, ultra thick] (9,0) -- (10,0); 

                \draw[blue, ultra thick] (0,3) arc (90:270:1.5);

                \draw[->, green, ultra thick] (13.5,3) -- (11.5,3);
                \draw[->, green, ultra thick] (5,3) -- (3,3);
            \end{tikzpicture}
            \caption{A diagram of several intermediate returns. The basepoint is represented by black circles and its opposite is represented by black rectangles.}
            \label{intermediatereturns}
        \end{figure}
        \noindent In the diagram, one can see a number of lines in \textcolor{red}{\textbf{red}}, \textcolor{blue}{\textbf{blue}}, and \textcolor{green}{\textbf{green}}. The lines in \textcolor{red}{\textbf{red}} represent states visited along a ``longer" intermediate return with no flips or stationary moves where including one is not optimal. Most of the states represented by the \textcolor{red}{\textbf{red}} lines are in the lower-right sector or the upper-left sector, where including a flip or stationary move will only serve to lengthen the intermediate return.\\

        \noindent Also, the lines in \textcolor{blue}{\textbf{blue}} represent intermediate returns that are ``too short". Put differently, when randomly selecting a state along our path to include a flip or stationary move while shortening, it is less likely that a flip or stationary move will be included in these intermediate returns, as opposed to the ``longer" intermediate returns.\\

        \noindent Finally, the lines in \textcolor{green}{\textbf{green}} represent states visited along a ``longer" intermediate return with no flips or stationary moves where including one is optimal. These states are located in the upper-right sector or the lower-left sector, where including a flip or stationary move will shorten the intermediate return by one of $\{1, 3, 5, \dots, 2d^{**}-1\}$ shift moves from either $\tilde{j}$ (for the upper-right sector) or from its opposite (for the lower-left sector).\\
        
        \noindent As future arguments will show, it is equally likely that a flip or stationary move can be included at any state represented by the \textcolor{green}{\textbf{green}} lines.
        \item It can be shown, by contradiction, that as $n \to \infty$ the number of states represented by \textcolor{green}{\textbf{green}} lines is equal to some positive constant multiple of the number of states represented by \textcolor{red}{\textbf{red}} lines. To see why, we select state $q = (\pm 1, q')$, where $q' \geq j^*$. Assume $d_2 - j^* \geq j^* - d_1$, and suppose $V_2$ is such that
        \begin{equation}
            P(V_2 > 2(10(d_2 - j^*)) + 1) > 0.4.
        \end{equation}
        \noindent Then
        \begin{align}
            \sum_{k = 10(d_2 - j^*) + 1}^{n - j^*}\left[(2k + 1)\left(1 - \frac{m_{j^* + k}(n)}{m_{j^*}(n)}\right)\right] &\geq 2(10(d_2 - j^*) + 1)\left(\frac{4}{10}\right)\\
            &= \frac{80(d_2 - j^*) + 8}{10}.
        \end{align}
        \noindent However,
        \begin{equation}
            \sum_{t_2 = j^*}^{d_2 - j^* + 1}\left[(2(t_2 - j^*) + 1)\left(1 - \frac{m_{t_2}(n)}{m_{j^*}(n)}\right)\right] \leq 2(d_2 - j^* + 1) + 1.
        \end{equation}
        \noindent This provides a contradiction since
        \begin{align}
            \sum_{t_2 = j^*}^{d_2 - j^* + 1}\left[(2(t_2 - j^*) + 1)\left(1 - \frac{m_{t_2}(n)}{m_{j^*}(n)}\right)\right] &\geq \frac{1}{2}\sum_{t_2 = j^*}^n\left[(2(t_2 - j^*) + 1)\left(1 - \frac{m_{t_2}(n)}{m_{j^*}(n)}\right)\right]\\
            &\geq \frac{80(d_2 - j^*) + 8}{10}.
        \end{align}
        \noindent Because of this contradiction, our conclusion can be drawn with respect to $d_2$. Similar arguments show the same contradiction is found with similar definitions for $V_1$ and $d_1$ if $j^* - d_1 \geq d_2 - j^*$.
        \item Using the definition of $d^{**}$ provided in Equation \eqref{d2stars}, and for $b_1 < b_2$, $(b_2 - b_1)d^{**} > 2(10d^{**}) + 1$, and with probability bounded below by a positive constant, at least one intermediate return will appear in some number $t$ steps, where $b_1d^{**} \leq t \leq b_2d^{**}$.

        \item By the properties of Markov chains, every intermediate return is independent of every other intermediate return. Because of this:
        \begin{enumerate}
            \item Shortening any intermediate return by including a flip or stationary move will not affect any other intermediate return,
            \item Because any state represented by \textcolor{green}{\textbf{green}} lines in Figure~\ref{intermediatereturns} is equally likely for a flip or stationary move to be included at while shortening an intermediate return:
            \begin{enumerate}
                \item The number of steps a single intermediate return is shortened by, after including a flip or stationary move, has a discrete uniform distribution on the discrete set $\{1, 3, 5, \dots, 2d^{**} - 1\}$ with probability $p = (d^{**})^{-1}$. A diagram of the probability mass function is shown in Figure~\ref{oneflip}.
                \begin{figure}[ht!]
                    \centering
                    \begin{tikzpicture}
                        \draw[<->, ultra thick] (0,0) -- (8.5,0);

                        \foreach \x in {0.5,0.75,1,1.25,1.5,1.75,2,2.25,2.5,2.75,3,3.25,3.5,3.75,4,4.25,4.5,4.75,5,5.25,5.5,5.75,6,6.25,6.5,6.75,7,7.25,7.5,7.75,8} {
                        \draw[red, fill = red] (\x,0.5) circle [radius = 0.1];
                        }
                        \draw (0.5,-4pt) -- (0.5,4pt) node[below=8pt] {$1$};
                        \draw (0.75,-4pt) -- (0.75,4pt) node[below=8pt] {$3$};
                        \draw (1,-4pt) -- (1,4pt) node[below=8pt] {$5$};
                        \draw (8,-4pt) -- (8,4pt) node[below=8pt] {$2d^{**}- 1$};
                        \draw (-4pt,0.5) -- (4pt,0.5) node[left=8pt] {$(d^{**})^{-1}$};
                    \end{tikzpicture}
                    \caption{A diagram of the probability mass function for shortening one intermediate return by including a flip or stationary move.}
                    \label{oneflip}
                \end{figure}
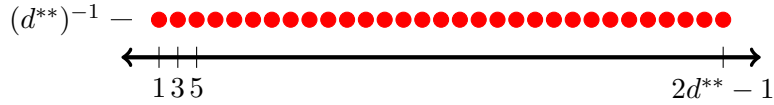
                \item The total number of steps two independent intermediate returns are shortened by, after including two flips and/or stationary moves, does not have a discrete uniform distribution on the discrete set $\{2, 4, 6, \dots, 4d^{**} - 2\}$. Instead, the distribution is a combination of two independent discrete uniform distributions on the same discrete set and, thus, is symmetric. Namely:
                \begin{enumerate}
                    \item The probability that the total number of steps the two intermediate returns are shortened by equals two is equal to $(d^{**})^{-2}$, which is the same as the probability that the total number of steps the two intermediate returns are shortened by equals $4d^{**} - 2$,
                    \item The probability that the total number of steps the two intermediate returns are shortened by equals four is equal to $2(d^{**})^{-2}$, which is the same as the probability that the total number of steps the two intermediate returns are shortened by equals $4d^{**}-4$,
                    \item A pattern exists where the probability that the two intermediate returns are shortened by $b$ steps is the same as the probability that the two intermediate returns are shortened by $4d^{**} - b$ steps.
                \end{enumerate}
                \noindent A diagram of the probability mass function is shown in Figure~\ref{twoflips}.
                \begin{figure}[ht!]
                    \centering
                    \begin{tikzpicture}
                        \draw[<->, ultra thick] (0,0) -- (8.5,0);

                        \draw[red, fill = red] (0.5,0.5) circle [radius = 0.1];
                        \draw[red, fill = red] (0.75,0.625) circle [radius = 0.1];
                        \draw[red, fill = red] (1,0.75) circle [radius = 0.1];
                        \draw[red, fill = red] (1.25,0.875) circle [radius = 0.1];
                        \draw[red, fill = red] (1.5,1) circle [radius = 0.1];
                        \draw[red, fill = red] (1.75,1.125) circle [radius = 0.1];
                        \draw[red, fill = red] (2,1.25) circle [radius = 0.1];
                        \draw[red, fill = red] (2.25,1.375) circle [radius = 0.1];
                        \draw[red, fill = red] (2.5,1.5) circle [radius = 0.1];
                        \draw[red, fill = red] (2.75,1.625) circle [radius = 0.1];
                        \draw[red, fill = red] (3,1.75) circle [radius = 0.1];
                        \draw[red, fill = red] (3.25,1.875) circle [radius = 0.1];
                        \draw[red, fill = red] (3.5,2) circle [radius = 0.1];
                        \draw[red, fill = red] (3.75,2.125) circle [radius = 0.1];
                        \draw[red, fill = red] (4,2.25) circle [radius = 0.1];
                        \draw[red, fill = red] (4.25,2.375) circle [radius = 0.1];
                        \draw[red, fill = red] (4.5,2.25) circle [radius = 0.1];
                        \draw[red, fill = red] (4.75,2.125) circle [radius = 0.1];
                        \draw[red, fill = red] (5,2) circle [radius = 0.1];
                        \draw[red, fill = red] (5.25,1.875) circle [radius = 0.1];
                        \draw[red, fill = red] (5.5,1.75) circle [radius = 0.1];
                        \draw[red, fill = red] (5.75,1.625) circle [radius = 0.1];
                        \draw[red, fill = red] (6,1.5) circle [radius = 0.1];
                        \draw[red, fill = red] (6.25,1.375) circle [radius = 0.1];
                        \draw[red, fill = red] (6.5,1.25) circle [radius = 0.1];
                        \draw[red, fill = red] (6.75,1.125) circle [radius = 0.1];
                        \draw[red, fill = red] (7,1) circle [radius = 0.1];
                        \draw[red, fill = red] (7.25,0.875) circle [radius = 0.1];
                        \draw[red, fill = red] (7.5,0.75) circle [radius = 0.1];
                        \draw[red, fill = red] (7.75,0.625) circle [radius = 0.1];
                        \draw[red, fill = red] (8,0.5) circle [radius = 0.1];

                        \draw (0.5,-4pt) -- (0.5,4pt) node[below=8pt] {$2$};
                        \draw (0.75,-4pt) -- (0.75,4pt) node[below=8pt] {$4$};
                        \draw (1,-4pt) -- (1,4pt) node[below=8pt] {$6$};
                        \draw (8,-4pt) -- (8,4pt) node[below=8pt] {$4d^{**} - 2$};
                    \end{tikzpicture}
                    \caption{A diagram of the probability mass function for shortening two intermediate returns by including flips and/or stationary moves.}
                    \label{twoflips}
                \end{figure}
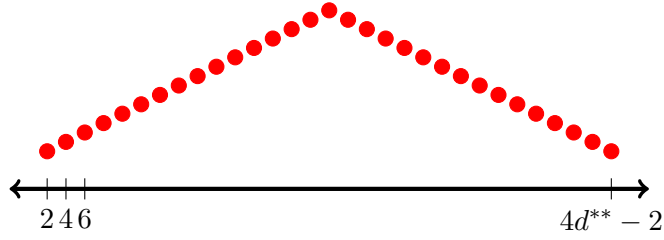
                \item The distribution of the total number of steps three independent intermediate returns are shortened by, after including three flips and/or stationary moves, behave similarly to the distribution of the total number of steps two independent intermediate returns are shortened by while shortening. The discrete set, however, is $\{3, 5, 7, \dots, 6d^{**} - 3\}$. Specifically:
                \begin{enumerate}
                    \item The probability that the total number of steps the three intermediate returns are shortened by equals three is equal to $(d^{**})^{-3}$, which is the same as the probability that the total number of steps the three intermediate returns are shortened by equals $6d^{**} - 3$,
                    \item The probability that the total number of steps the three intermediate returns are shortened by equals five is equal to $3(d^{**} )^{-3}$, which is the same as the probability that the total number of steps the three intermediate returns are shortened by equals $6d^{**} -5$,
                    \item A pattern exists where the probability that the three intermediate returns are shortened by $b$ steps is the same as the probability that the three intermediate returns are shortened by $6d^{**}  - b$ steps.
                \end{enumerate}
                \noindent A diagram of the probability mass function is shown in Figure~\ref{threeormoreflips}.
                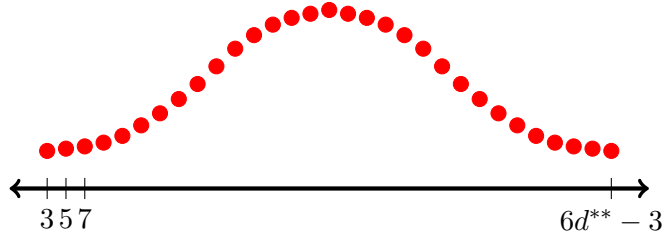
\begin{figure}[ht!]
                    \centering
                    \begin{tikzpicture}
                        \draw[<->, ultra thick] (0,0) -- (8.5,0);

                        \draw[red, fill = red] (0.5,0.5) circle [radius = 0.1];
                        \draw[red, fill = red] (0.75,0.53) circle [radius = 0.1];
                        \draw[red, fill = red] (1,0.56) circle [radius = 0.1];
                        \draw[red, fill = red] (1.25,0.61) circle [radius = 0.1];
                        \draw[red, fill = red] (1.5,0.7) circle [radius = 0.1];
                        \draw[red, fill = red] (1.75,0.84) circle [radius = 0.1];
                        \draw[red, fill = red] (2,1) circle [radius = 0.1];
                        \draw[red, fill = red] (2.25,1.19) circle [radius = 0.1];
                        \draw[red, fill = red] (2.5,1.39) circle [radius = 0.1];
                        \draw[red, fill = red] (2.75,1.625) circle [radius = 0.1];
                        \draw[red, fill = red] (3,1.86) circle [radius = 0.1];
                        \draw[red, fill = red] (3.25,2.04) circle [radius = 0.1];
                        \draw[red, fill = red] (3.5,2.18) circle [radius = 0.1];
                        \draw[red, fill = red] (3.75,2.27) circle [radius = 0.1];
                        \draw[red, fill = red] (4,2.325) circle [radius = 0.1];
                        \draw[red, fill = red] (4.25,2.375) circle [radius = 0.1];
                        \draw[red, fill = red] (4.5,2.325) circle [radius = 0.1];
                        \draw[red, fill = red] (4.75,2.27) circle [radius = 0.1];
                        \draw[red, fill = red] (5,2.18) circle [radius = 0.1];
                        \draw[red, fill = red] (5.25,2.04) circle [radius = 0.1];
                        \draw[red, fill = red] (5.5,1.86) circle [radius = 0.1];
                        \draw[red, fill = red] (5.75,1.625) circle [radius = 0.1];
                        \draw[red, fill = red] (6,1.39) circle [radius = 0.1];
                        \draw[red, fill = red] (6.25,1.19) circle [radius = 0.1];
                        \draw[red, fill = red] (6.5,1) circle [radius = 0.1];
                        \draw[red, fill = red] (6.75,0.84) circle [radius = 0.1];
                        \draw[red, fill = red] (7,0.7) circle [radius = 0.1];
                        \draw[red, fill = red] (7.25,0.61) circle [radius = 0.1];
                        \draw[red, fill = red] (7.5,0.56) circle [radius = 0.1];
                        \draw[red, fill = red] (7.75,0.53) circle [radius = 0.1];
                        \draw[red, fill = red] (8,0.5) circle [radius = 0.1];

                        \draw (0.5,-4pt) -- (0.5,4pt) node[below=8pt] {$3$};
                        \draw (0.75,-4pt) -- (0.75,4pt) node[below=8pt] {$5$};
                        \draw (1,-4pt) -- (1,4pt) node[below=8pt] {$7$};
                        \draw (8,-4pt) -- (8,4pt) node[below=8pt] {$6d^{**} - 3$};
                    \end{tikzpicture}
                    \caption{A diagram of the probability mass function for shortening three intermediate returns by including flips and/or stationary moves.}
                    \label{threeormoreflips}
                \end{figure}
                \item A pattern exists for shortening $a'$ intermediate returns by including $a'$ flips and/or stationary moves. For the first $a' - 1$ shortened intermediate returns, we can choose at least $\tilde{a}d^{**}/(2a')$ values from $\{1, 3, 5, \dots, 2d^{**} - 1\}$ such that the last intermediate return to be shortened will also be shortened by a length whose value is in $\{1, 3, 5, \dots, 2d^{**} - 1\}$. In addition:
                \begin{itemize}
                    \item There is a parity issue that needs to be considered. If $a'$ is an odd integer, then the length of the overshoot must be an odd integer. Likewise, if $a'$ is an even integer, then the length of the overshoot must be an even integer.
                    \item The discrete set of possible values for the total number of steps $a'$ intermediate returns are shortened by is $\{a', a' + 2, a' + 4, \dots, a'(2d^{**} - 1)\}$.
                \end{itemize}
                \noindent A diagram of the probability mass function is shown in Figure~\ref{aflips}.
                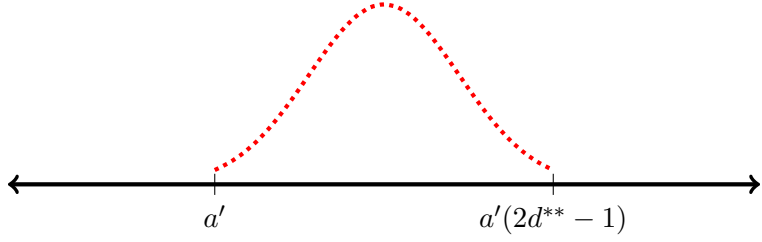
\begin{figure}[ht!]
                    \centering
                    \begin{tikzpicture}
                        \draw[red, smooth, dotted, ultra thick, domain = 2.75:7.25] plot (\x, {6*exp((-1)*(\x - 5)*(\x - 5)/2)/sqrt(2*pi)});

                        \draw[<->, ultra thick] (0,0) -- (10,0);

                        \draw (2.75,-4pt) -- (2.75,4pt) node[below=8pt] {$a'$};
                        \draw (7.25,-4pt) -- (7.25,4pt) node[below=8pt] {$a'(2d^{**} - 1)$};
                    \end{tikzpicture}
                    \caption{A diagram of the probability mass function for shortening $a'$ intermediate returns by including flips and/or stationary moves.}
                    \label{aflips}
                \end{figure}
            \end{enumerate}
        \end{enumerate}
        \noindent Because of these, we may conclude that the number of ways to shorten $a'$ intermediate returns by a possible total number of steps in $[\tilde{a}d^{**}, (a' - \tilde{a})d^{**}]$, where $0 < \tilde{a} < a'/2$, is bounded below by some positive constant, $c'''$, multiple of $\left(d^{**}\right)^{a' - 1}$.
    \end{enumerate}

    \subsubsection{Step 2} \label{sectionstep2}

    \noindent We now proceed to Step 2, the calculation of a lower bound on the occurrence probability of the constructed path from Step 1. Let $\{X_{j''}\}_{j'' = 0}^{cn}$ be the collection of states of $\Phi$ that are attained at each step (with $X_0 = y$ and $X_{cn} = j'$). We begin with finding a lower bound on the occurrence probability of the first component.
    \begin{propn}
        Assume the requirements of Lemma~\ref{step1} hold true, and suppose $E$ is the event that occurs when no flips or stationary moves occur in the first $2n - 1$ steps. Then
        \begin{equation}
            \sum_{t = 0}^{2n - 1}\Biggl[P(X_t = \tilde{j}\wedge X_l \neq \tilde{j}\text{ }\forall\text{ }l < t\text{ $|$ }X_0 = y)\Biggr] \geq P(E) \geq \left(\frac{n - 1}{n}\right)^{2n - 1}. \label{3.6}
        \end{equation} \label{1stcomponent}
    \end{propn}
    \noindent \textit{Proof.} The left-hand side of Equation \eqref{3.6} is bounded below by $P(E)$. Since the probability of a flip or stationary move, at any step, is on the order of $n^{-1}$, the right-hand side of Equation \eqref{3.6} is proved. We are done. \qed\\

    \noindent We next state a proposition that will give a lower bound on the occurrence probability of the third component.
    \begin{propn}
        Assume the requirements of Lemma~\ref{step1} hold true.  Suppose
        \begin{equation}
            \{G_t\}_{t = 0}^{cn} = \left\{X_t = j'\wedge X_t \neq \tilde{j}\text{ }\forall\text{ } t > 0\wedge X_l \neq j'\text{ }\forall\text{ }0 < l < t\text{ $|$ }X_0 = \tilde{j}\right\}_{t = 0}^{cn}.
        \end{equation}
        \noindent Then if $j^{\prime}=(+1,j)$ or $j^{\prime}=(-1,j)$,
        \begin{equation}
            \sum_{t = 0}^{2n - 1}\Biggl[P(G_t)\Biggr] \geq \dfrac{m_j(n)}{m_{j^*}(n)}\left(\dfrac{n - 1}{n}\right)^{2n - 1}. \label{3.7}
        \end{equation} \label{3rdcomponent}
    \end{propn}
    \noindent \textit{Proof.} Here, the argument used in Proposition~\ref{1stcomponent} needs to be adjusted. With $\pi$ as constructed, it is possible that $\Phi$ will take more steps than necessary to move from $\tilde{j}$ to $j'$. Because of this, the ratio seen in the right-hand side of Equation \eqref{3.7} is included to account for this. This proves the proposition, and we are done. \qed\\
    
    \noindent We next need to find a lower bound on the probability of the second component. Here, let $F'$ be the random variable that represents the number of flips or stationary moves in a single intermediate return, and let $F$ be the random variable that represents the number of flips or stationary moves in all intermediate returns prior to shortening. We start with the proposition to follow.
    \begin{propn}
        Assume the requirements of Lemma~\ref{step1} hold true. Then
        \begin{equation}
            P(F' = 0) \geq \left(\frac{n - 1}{n}\right)^{2n}.\label{3.8}
        \end{equation} \label{proposition3}
    \end{propn}
    \noindent \textit{Proof.} It is not enough to say that the longest intermediate return with no flips or stationary moves has length $2n$. Instead, we must sum over all possibilities for intermediate returns. Doing this gives
    \begin{equation}
        P(F' = 0) = \sum_{t_2 = 0}^{j^* - 1}\sum_{t_1 = 0}^{n - j^*}\Biggl[\left(\tfrac{n - 1}{n}\right)^{2t_1 + 2t_2 + 2}\left(\tfrac{m_{j^* + t_1}(n)}{m_{j^*}(n)}\right)\left(1 - \tfrac{m_{j^* + t_1 + 1}(n)}{m_{j^* + t_1}(n)}\right)\left(\tfrac{m_{j^* - t_2}(n)}{m_{j^*}(n)}\right)\left(1 - \tfrac{m_{j^* - t_2 - 1}(n)}{m_{j^* - t_2}(n)}\right)\Biggr],
    \end{equation}
    \noindent which simplifies to
    \begin{align}
        P(F' = 0) &= \sum_{t_2 = 0}^{j^* - 1}\sum_{t_1 = 0}^{n - j^*}\Biggl[\left(\frac{n - 1}{n}\right)^{2t_1 + 2t_2 + 2}\left(\frac{m_{j^* + t_1}(n) - m_{j^* + t_1 + 1}(n)}{m_{j^*}(n)}\right)\left(\frac{m_{j^* - t_2}(n) - m_{j^* - t_2 - 1}(n)}{m_{j^*}(n)}\right)\Biggr]\\
        &\geq \left(\frac{1}{m_{j^*}(n)}\right)^2\left(\frac{n - 1}{n}\right)^{2n}\sum_{t_2 = 0}^{j^* - 1}\sum_{t_1 = 0}^{n - j^*}\Biggl[\left(m_{j^* + t_1}(n) - m_{j^* + t_1 + 1}(n)\right)\left(m_{j^* - t_2}(n) - m_{j^* - t_2 - 1}(n)\right)\Biggr]\\
        &\geq \frac{\left(m_{j^*}(n) - m_{n + 1}(n)\right)\left(m_{j^*}(n) - m_0(n)\right)}{\left(m_{j^*}(n)\right)^2}\left(\frac{n - 1}{n}\right)^{2n}. \label{3.9}
    \end{align}
    \noindent Defining
    \begin{equation}
        m_0(n) = m_{n + 1}(n) = 0, \label{3.10}
    \end{equation}
    \noindent applying Equation \eqref{3.10} to Equation \eqref{3.9}, and simplifying gives
    \begin{equation}
        P(F' = 0) \geq \left(\frac{n - 1}{n}\right)^{2n}. \label{3.11}
    \end{equation}
    \noindent Since Equation \eqref{3.11} agrees with Equation \eqref{3.8}, we are done. \qed\\

    \noindent Similar arguments prove the proposition below.
    \begin{propn}
        Assume the requirements of Lemma~\ref{step1} hold true. Then
        \begin{equation}
            P(F' = 1) \geq \frac{1}{n}\left(\frac{n - 1}{n}\right)^{4n - 2j^* - 2}. \label{3.12}
        \end{equation} \label{proposition4}
    \end{propn}
    \noindent Equation \eqref{3.8} and Equation \eqref{3.12} are not directly used in the calculation of a lower bound on the occurrence probability of a path of length $x = cn$ in the general unimodal case. Instead, these are stated to show that we will not need to rely on all values $\pi$ can take. Furthermore, they allow us to find a lower bound on the path probability of the second component, which is stated and proved below.
    \begin{propn}
        Assume the requirements of Lemma~\ref{step1} hold true. Then
        \begin{equation}
            P(F = f_1) \geq \min_{s = cn - 4n + 2}^{cn}\left[\dbinom{s}{f_1}\left(\frac{1}{n}\right)^{f_1}\left(\frac{n - 1}{n}\right)^{s - f_1}\right]. \label{3.13}
        \end{equation} \label{2ndcomponent1stpiece}
    \end{propn}
    \noindent \textit{Proof.}  At any step in the second component, the algorithm that defines the Diaconis-Holmes-Neal sampler can either have the path flip (or stay put) or not flip.  Because the occurrence of a flip or stationary move at any point in the second component is independent of the occurrence of a flip or stationary move at any other point in the second component, $F$ has a binomial distribution with independent trials parameter equal to the number of steps in the second component and probability parameter $p = n^{-1}$.\\

    \noindent One question, however, that remains to be answered is the value of the independent trials parameter.  By the construction of the first and third components, the number of steps the second component will require could be as small as $cn - 4n + 2$ or as large as $cn$.  Because of this, a lower bound on $P(F = f_1)$ will take the minimum probability over the possible values for the independent trials parameter. This completes the proof, and we are done. \qed\\

    \noindent \textbf{Note.} Because the minimum of Equation \eqref{3.13} occurs when $s = cn$, we may simplify Equation \eqref{3.13} to
    \begin{equation}
        P(F = f_1) \geq \dbinom{cn}{f_1}\left(\frac{1}{n}\right)^{f_1}\left(\frac{n - 1}{n}\right)^{cn - f_1}. \label{3.14}
    \end{equation}
    \noindent In order to find a lower bound on the occurrence probability of the path constructed in Lemma~\ref{step1}, we need two more things. We need a multiplier for including flips and stationary moves while shortening, and we need a multiplier for the probability that the constructed path has length $x = cn$. The multiplier for including flips and stationary moves while shortening is established in the following.
    \begin{propn}
        Suppose $F''$ is the random variable that equals the number of additional flips or stationary moves included in the intermediate returns while shortening the path from $y$ to $j'$. Then $F'' \sim b\left(\lfloor cn/\tilde{f}\rfloor, n^{-1}\right)$. \label{2ndcomponent2ndpiece}
    \end{propn}
    \noindent \textit{Proof.} In the proof of the previous proposition, we stated that ``the algorithm that defines the Diaconis-Holmes-Neal sampler can either have the path flip (or stay put) or not flip". We also stated that, because of this, the number of flips and stationary moves in the second component has a binomial distribution. Here, the probability parameter of the binomial distribution remains the same, $n^{-1}$. All that is left to argue is that the independent trials parameter is equal to $\lfloor cn/\tilde{f}\rfloor$.\\

    \noindent Because the sampler can either have the path flip or not flip, each state is equally likely to have a flip or stationary move. This was mentioned in Note 2 to the proof of Lemma~\ref{step1}. However, half of the states (namely, the states in the lower-right quadrant and the upper-left quadrant) are not ``good" locations for including a flip or stationary move to shorten an intermediate return. This is because including a flip or stationary move in any of these states will only lengthen the intermediate return and, thus, increase the length of the overshoot.\\

    \noindent In addition, per Note 2 to the proof of Lemma~\ref{step1} and with probability bounded above by a positive constant less than one, the sampler is in a state represented by \textcolor{red}{\textbf{red}} or \textcolor{blue}{\textbf{blue}} lines (as shown in Figure~\ref{intermediatereturns}). In other words, the sampler is in a state where it is not optimal to include a flip or stationary move. Thus, with probability bounded below by the positive constant $\tilde{f}^{-1}$, the sampler is in a state represented by \textcolor{green}{\textbf{green}} lines (as shown in Figure~\ref{intermediatereturns}) and we can shorten the intermediate return by an odd number of steps between one and $2d^{**} + 1$, with each possibility equally likely. Because of this, ideally we want the independent trials parameter to be equal to $(cn)/\tilde{f}$. However, since this can take on non-integer values, taking the floor of the ideal independent trials parameter is sufficient.\qed\\

    \noindent Finding the multiplier for the probability that the constructed path has exactly $cn$ steps, however, requires some work. First, we show the following.
    \begin{propn}
        The Diaconis-Holmes-Neal sampler, in the general unimodal case, is non-null and persistent. \label{non-null}
    \end{propn}
    \noindent \textit{Proof.} Because a path can be constructed between any two states in the Diaconis-Holmes-Neal sampler, it is irreducible. In addition, Theorem 9.1 of Guhaniyogi and Kang \cite{kang} states that every irreducible Markov chain with a stationary distribution is non-null and persistent. We are done. \qed\\

    \noindent Next, we state the following theorem from p. 227 of Grimmett and Stirzaker \cite{grimmett}.
    \begin{theorem}
        An irreducible chain has a stationary distribution if, and only if, all states are non-null and persistent; in this case $\tilde{\pi}$ is the stationary distribution and is given by $\tilde{\pi}(j) = \mu_j^{-1}$ for each $j$, where $\mu_j$ is the mean recurrence time of state $j$. \label{recurrence1}
    \end{theorem}
    \noindent Applying this theorem to the general symmetric case, we immediately obtain the following.
    \begin{corol}
        The mean length of an intermediate return in the general unimodal case is
        \begin{equation}
            \mu_{\tilde{j}} = \frac{2z'}{m_{j^*}(n)}. \label{3.15}
        \end{equation} \label{recurrence2}
    \end{corol}
    \noindent Before proceeding to the next proposition, recall Note 3, Note 4, and Note 5 to the proof of Lemma~\ref{step1} allow us to conclude that, with probability bounded below by a positive constant, the number of intermediate returns contained in any overshoot is at least one and the length of any overshoot is some constant, $\tilde{c}$, multiple of $d^{**}$. Using these and Chebyshev's inequality, applied to the mean length of each of $r(n)$ intermediate returns contained in any overshoot, allow us to state the following.
    \begin{propn}
        Suppose $\{R_l\}_{l = 1}^{r(n)}$ is the sequence of return times (in steps) of the intermediate returns contained in any overshoot. Then, with probability bounded below by a positive constant and for some constant $\tilde{c} > 4$,
        \begin{equation}
            r(n) < \frac{\tilde{c}d^{**}m_{j^*}(n)}{4z'} \label{3.16}
        \end{equation} \label{returntimes}
    \end{propn}
    \noindent \textit{Proof.} Applying Chebyshev's inequality to the mean length of $r(n)$ intermediate returns contained in any overshoot, we see that
    \begin{align}
        P\left(\sum_{l = 1}^{r(n)}R_l > \tilde{c}d^{**}\right) \leq \frac{\mu_{\tilde{j}}r(n)}{\tilde{c}d^{**}}. \label{3.17}
    \end{align}
    \noindent Here, $\tilde{c}$ is any constant, not depending on $n$, $d_1$, $d_2$, or the underlying probabilities, and $\tilde{j}$ is the basepoint. Bounding Equation \eqref{3.17} by 50\% and solving for $r(n)$ yields
    \begin{align}
        \frac{\mu_{\tilde{j}}r(n)}{\tilde{c}d^{**}} &< \frac{1}{2} & &\longleftrightarrow & r(n) &< \frac{\tilde{c}d^{**}}{2\mu_{\tilde{j}}} \label{3.18}\\
        & & &\longleftrightarrow & r(n) &< \frac{\tilde{c}d^{**}\tilde{\pi}(\tilde{j})}{2} \label{3.19}\\
        & & &\longleftrightarrow & r(n) &< \frac{\tilde{c}d^{**}m_{\tilde{j}}(n)}{4z'}\\
        & & &\longleftrightarrow & r(n) &< \frac{\tilde{c}d^{**}m_{j^*}(n)}{4z'} \label{3.20}
    \end{align}
    \noindent Since Equation \eqref{3.20} agrees with Equation \eqref{3.16}, we are done. \qed\\

    \noindent The results from Proposition~\ref{non-null}, Corollary~\ref{recurrence2}, and Proposition~\ref{returntimes} allow us to obtain the result shown below.
    \begin{propn}
        The probability that the path, as constructed in Lemma~\ref{step1}, has exactly $cn$ steps is bounded below by
        \begin{equation}
            P' \geq \frac{\tilde{\tilde{c}}m_{j^*}(n)}{z'}, \label{3.21}
        \end{equation}
        \noindent where $\tilde{\tilde{c}}$ is a constant. \label{2ndcomponent3rdpiece}
    \end{propn}
    \noindent \textit{Proof.} Note 3 and Note 4 to the proof of Lemma~\ref{step1} give that the probability of having at least one intermediate return contained in any overshoot is larger than a positive constant. Furthermore, Equation \eqref{3.17} implies that the probability of having at least $\tilde{c}d^{**}m_{j^*}(n)/(4z')$ intermediate returns in $\tilde{c}d^{**}$ steps from $\tilde{j}$ is at least $c''$. With this,
    \begin{align}
        P' &\geq \frac{\tilde{c}c''d^{**}m_{j^*}(n)}{4d^{**}z'} \label{3.22}\\
        &= \frac{\tilde{\tilde{c}}d^{**}m_{j^*}(n)}{d^{**}z'} \label{3.23}\\
        &= \frac{\tilde{\tilde{c}}m_{j^*}(n)}{z'}. \label{3.24}
    \end{align}
    \noindent Here, the jump from Equation \eqref{3.22} to Equation \eqref{3.23} is made by setting $\tilde{\tilde{c}} = \tilde{c}c''/4$. Since Equation \eqref{3.24} agrees with Equation \eqref{3.21}, we are done. \qed\\

    \noindent \textbf{Note.} The multiplier of $d^{**}$ found in the denominator of Equation \eqref{3.22} is a direct consequence of Note 5 to the proof of Lemma~\ref{step1}.\\

    \noindent The results proved in this section allow us to find a lower bound on the occurrence probability of a path in the general unimodal case, which we state below.
    \begin{lemma}
        Assume the requirements of Lemma~\ref{step1} hold true, let $f_1$ be the number of flips and stationary moves in any path of length $x = cn$ before shortening, and let $f_2$ be the number of flips and stationary moves included in the path while shortening. Then the path has occurrence probability bounded below by
        \begin{equation}
            \tilde{p} \geq \frac{\tilde{\tilde{c}}m_j(n)}{z'}\dbinom{cn}{f_1}\dbinom{\lfloor cn/\tilde{f}\rfloor}{f_2}\left(\frac{1}{n}\right)^{f_1 + f_2}\left(\frac{n - 1}{n}\right)^{cn + 4n - f_1 - (2d^{**} + 1)f_2 - 2}. \label{3.25}
        \end{equation} \label{step2}
    \end{lemma}
    \noindent \textit{Proof.} The result is obtained by multiplying Equation \eqref{3.6}, Equation \eqref{3.7}, Equation \eqref{3.14}, the multiplier for including flips and stationary moves while shortening, and Equation \eqref{3.21}. Multiplying these together yields
    \begin{equation}
        \tilde{p} = P(E)P(F = f_1)P(F'' = f_2)P'\sum_{t = 0}^{2n - 1}\Biggl[P(G_t)\Biggr].
    \end{equation}
    \noindent Applying the computed lower bounds, we obtain
    \begin{align}
        \tilde{p} &\geq \frac{m_j(n)}{m_{j^*}(n)}\left(\frac{n - 1}{n}\right)^{4n - 2}\dbinom{cn}{f_1}\left(\frac{1}{n}\right)^{f_1}\left(\frac{n - 1}{n}\right)^{cn - f_1}\dbinom{\lfloor cn/\tilde{f}\rfloor}{f_2}\left(\frac{1}{n}\right)^{f_2}\left(\frac{n - 1}{n}\right)^{-(2d^{**} + 1)f_2}\left(\frac{\tilde{\tilde{c}}m_{j^*}(n)}{z'}\right) \label{3.26}\\
        &= \frac{\tilde{\tilde{c}}m_j(n)}{z'}\dbinom{cn}{f_1}\dbinom{\lfloor cn/\tilde{f}\rfloor}{f_2}\left(\frac{1}{n}\right)^{f_1 + f_2}\left(\frac{n - 1}{n}\right)^{4n - 2 + cn - f_1 - (2d^{**} + 1)f_2}\\
        &= \frac{\tilde{\tilde{c}}m_j(n)}{z'}\dbinom{cn}{f_1}\dbinom{\lfloor cn/\tilde{f}\rfloor}{f_2}\left(\frac{1}{n}\right)^{f_1 + f_2}\left(\frac{n - 1}{n}\right)^{cn + 4n - f_1 - (2d^{**} + 1)f_2 - 2}. \label{3.27}
    \end{align}
    \noindent Since Equation \eqref{3.27} agrees with Equation \eqref{3.25}, we are done. \qed\\

    \noindent \textbf{Note.} The multiplier that represents $P(F'' = f_2)$ is expressed as is in Equation \eqref{3.26} because it represents shortening our path from length $cn + b$ to length $cn$, where $b$ is the length of any overshoot.

    \subsubsection{Step 3} \label{sectionstep3}
    
    \noindent We now proceed to Step 3, the calculation of a lower bound on the probability that a path of length $x = cn$ starts at $y$ and ends in some subset of states of the Markov chain. Here, however, only two adjustments to Equation \eqref{3.25} are required. The proposition and corollary that follow will give us the adjustments to make.
    \begin{propn}
        Suppose $a$ is a constant, independent of $n$.  Then
        \begin{equation}
            \dbinom{n}{a} \lesssim n^a. \label{3.28}
        \end{equation} \label{notation1}
    \end{propn}
    \noindent \textit{Proof.}
    \begin{equation}
        \dbinom{n}{a} = \frac{n!}{(n - a)!a!},
    \end{equation}
    \noindent which, by the properties of factorials, reduces to
    \begin{equation}
        \dbinom{n}{a} = \prod_{j = 0}^{a - 1}\Biggl[\frac{n - j}{j + 1}\Biggr]. \label{3.29}
    \end{equation}
    \noindent Since $a$ is independent of $n$, Equation \eqref{3.29}, after simplifying, will have a $n^a$ term but no power of $n$ larger than $a$. Using notation given in both function of $n$ cases, we have Equation \eqref{3.28}.  We are done.\qed\\

    \noindent An immediate consequence of this proposition is the following corollary.
    \begin{corol}
        Both
        \begin{equation}
            \dbinom{cn}{f_1} \lesssim (cn)^{f_1} \label{3.30}
        \end{equation}
        \noindent and
        \begin{equation}
            \dbinom{\lfloor cn/\tilde{f}\rfloor}{f_2} \lesssim \left(\frac{cn}{\tilde{f}}\right)^{f_2}. \label{3.31}
        \end{equation}
        \noindent are true. \label{notation2}
    \end{corol}
    \noindent \textit{Proof.} Equation \eqref{3.29} becomes
    \begin{align}
        \dbinom{cn}{f_1} &= \prod_{j = 0}^{f_1 - 1}\Biggl[\frac{cn - j}{j + 1}\Biggr]\\
        &= \frac{cn(cn - 1)(cn - 2)\dots(cn - f_1 + 1)}{f_1(f_1 - 1)(f_1 - 2)\dots(3)(2)(1)}\\
        &\lesssim cn(cn - 1)(cn - 2)\dots (cn - f_1 + 1)\\
        &\lesssim (cn)^{f_1}. \label{3.32}
    \end{align}
    \noindent Since Equation \eqref{3.32} agrees with Equation \eqref{3.30}, we have proved the first half of the corollary. For the second half, we see that
    \begin{align}
        \dbinom{\lfloor cn/\tilde{f}\rfloor}{f_2} &= \prod_{j = 0}^{f_2}\left[\frac{\left\lfloor\frac{cn}{\tilde{f}} - j\right\rfloor}{j + 1}\right]\\
        &= \frac{\left\lfloor\frac{cn}{\tilde{f}}\right\rfloor\left(\left\lfloor\frac{cn}{\tilde{f}}\right\rfloor - 1\right)\left(\left\lfloor\frac{cn}{\tilde{f}}\right\rfloor - 2\right)\cdots\left(\left\lfloor\frac{cn}{\tilde{f}}\right\rfloor - f_2 + 1\right)}{f_2(f_2 - 1)(f_2 - 2)\cdots(3)(2)(1)}\\
        &\lesssim \left\lfloor\frac{cn}{\tilde{f}}\right\rfloor\left(\left\lfloor\frac{cn}{\tilde{f}}\right\rfloor - 1\right)\left(\left\lfloor\frac{cn}{\tilde{f}}\right\rfloor - 2\right)\cdots\left(\left\lfloor\frac{cn}{\tilde{f}}\right\rfloor - f_2 + 1\right)\\
        &\lesssim \left\lfloor\frac{cn}{\tilde{f}}\right\rfloor^{f_2}.
    \end{align}
    \noindent Because $\lfloor k\rfloor \leq k$ and $\lfloor k\rfloor^\alpha \leq k^\alpha$ for any $k$ and $\alpha$,
    \begin{equation}
        \dbinom{\lfloor cn/\tilde{f}\rfloor}{f_2} \lesssim \left(\frac{cn}{\tilde{f}}\right)^{f_2}. \label{3.33}
    \end{equation}
    \noindent Since this agrees with Equation \eqref{3.31}, we have proved the second half of the corollary. We are done.\qed\\

    \noindent Combining Equation \eqref{3.25}, Equation \eqref{3.30}, and Equation \eqref{3.31} allows us to state the following without proof.
    \begin{lemma}
        Assume the requirements of Lemma~\ref{step1} hold true. Then the probability of $\Phi$ starting at state $y$ and ending at state $j'$, where $j' = (\epsilon, j)$ with $\epsilon = \pm 1$ and $1 \leq j \leq n$, is such that
        \begin{equation}
            \tilde{\tilde{p}} \gtrsim \frac{\tilde{\tilde{c}}c^{f_1 + f_2}m_j(n)}{\tilde{f}^{f_2}z'}\left(\frac{n - 1}{n}\right)^{cn + 4n - f_1 - (2d^{**} + 1)f_2 - 2}. \label{3.34}
        \end{equation} \label{step3}
    \end{lemma}

    \subsubsection{Steps 4 and 5} \label{sectionstep4and5}

    \noindent Finally, we proceed to Steps 4 and 5, the step to apply Theorem~\ref{MeynTweedie2009Thm} to Equation \eqref{3.34} and take a limit as $n \to \infty$ of the result from Step 4.\\

    \noindent Since Equation \eqref{3.34} gives a lower bound on the probability of a path in $\Phi$ starting at $y$ and ending in any subset $A$ of ${\cal B}(\Phi)$, Theorem~\ref{MeynTweedie2009Thm} can be applied to find an upper bound on the variation distance between the lower bound computed in Equation \eqref{3.34} and the stationary distribution of $\Phi$. First, we must compute $\rho$ as specified in Theorem~\ref{MeynTweedie2009Thm}. However, because there are up to $n$ possible values that the underlying probabilities can take, the calculation of $\rho$ must be modified from the calculations shown in both function of $n$ cases.\\

    \noindent The modification, however, is simple. For each $1 \leq k \leq n$, define $A_k = \{(+1, k), (-1, k)\}$. Then
    \begin{align}
        \rho &\lesssim 1 - \sum_{j = 1}^n\left[\frac{|A_k|\tilde{\tilde{c}}c^{f_1 + f_2}m_j(n)}{\tilde{f}^{f_2}z'}\left(\frac{n - 1}{n}\right)^{cn + 4n - f_1 - (2d^{**} + 1)f_2 - 2}\right]\\
        &= 1 - 2\sum_{j = 1}^n\Biggl[\frac{\tilde{\tilde{c}}c^{f_1 + f_2}m_j(n)}{\tilde{f}^{f_2}z'}\left(\frac{n - 1}{n}\right)^{cn + 4n - f_1 - (2d^{**} + 1)f_2 - 2}\Biggr]\\
        &= 1 - \frac{2\tilde{\tilde{c}}c^{f_1 + f_2}}{\tilde{f}^{f_2}z'}\left(\frac{n - 1}{n}\right)^{cn + 4n - f_1 - (2d^{**} + 1)f_2 - 2}\sum_{j = 1}^n\Biggl[m_j(n)\Biggr]\\
        &= 1 - \frac{2\tilde{\tilde{c}}c^{f_1 + f_2}z'}{\tilde{f}^{f_2}z'}\left(\frac{n - 1}{n}\right)^{cn + 4n - f_1 - (2d^{**} + 1)f_2 - 2},
    \end{align}
    \noindent which simplifies further to
    \begin{align}
        \rho &\lesssim 1 - \frac{2\tilde{\tilde{c}}c^{f_1 + f_2}}{\tilde{f}^{f_2}}\left(\frac{n - 1}{n}\right)^{cn + 4n - f_1 - (2d^{**} + 1)f_2 - 2}. \label{3.35}
    \end{align}
    \noindent Applying the substitutions
    \begin{equation}
        \hat{c} = \frac{2\tilde{\tilde{c}}c^{f_1 + f_2}}{\tilde{f}^{f_2}}
    \end{equation}
    \noindent and
    \begin{equation}
        cn - f_1 - f_2 = cn + 4n - f_1 - (2d^{**} + 1)f_2 - 2 \label{3.36}
    \end{equation}
    \noindent to Equation \eqref{3.35} yields
    \begin{equation}
        \rho \lesssim 1 - \hat{c}\left(\frac{n - 1}{n}\right)^{cn - f_1 - f_2}. \label{3.37}
    \end{equation}
    \noindent The substitution given in Equation \eqref{3.36} is made by recognizing that the number of steps in a path of length $x = cn$ that are not flips or stationary moves is equal to $cn - f_1 - f_2$. From here, applying Equation \eqref{3.37} to Equation \eqref{1.1} yields
    \begin{align}
        ||P^z(y,A) - \tilde{\pi}|| &\leq \rho^{\lfloor z/x \rfloor} \label{3.38}\\
        &\lesssim \left(1 - \hat{c}\left(\frac{n - 1}{n}\right)^{cn - f_1 - f_2}\right)^{\lfloor z/x \rfloor} \label{3.39}\\
        &= \left(1 - \hat{c}\left(\frac{n - 1}{n}\right)^{cn - f_1 - f_2}\right)^{\lfloor (c'n)/(cn) \rfloor} \label{3.40}\\
        &= \left(1 - \hat{c}\left(\frac{n - 1}{n}\right)^{cn - f_1 - f_2}\right)^{\lfloor c'/c \rfloor} \label{3.41}.
    \end{align}
    \noindent \textbf{Note.} The jump from Equation \eqref{3.39} to Equation \eqref{3.40} is made by letting $z = c'n$ for some constant $c'$ that is independent of $n$. This represents any generic path in $\Phi$ of length $c'n$.\\

    \noindent Equation \eqref{3.41} completes Step 4, and we can now proceed to Step 5. Taking its limit as $n \to \infty$ yields
    \begin{align}
        \lim_{n \to \infty}\Biggl[||P^z(y,A) - \tilde{\pi}||\Biggr] &\lesssim \lim_{n \to \infty}\left[\left(1 - \hat{c}\left(\frac{n - 1}{n}\right)^{cn - f_1 - f_2}\right)^{\lfloor c'/c \rfloor}\right]\\
        &= \left(1 - \hat{c}e^{-c}\right)^{\lfloor c'/c \rfloor}. \label{3.42}
    \end{align}
    \noindent Since Equation \eqref{3.42} agrees with Equation \eqref{3.2}, we have proved Theorem~\ref{mainresult}. \qed

    \section{Conclusion and Questions for Further Study} \label{conclusion}

    \noindent By improving the number of steps needed to converge to stationarity, the Diaconis-Holmes-Neal Markov chain sampler has significantly advanced applications of Metropolis et. al \cite{Metropolis1953}. \cite{DHN2000} Work presented by Hildebrand (\cite{Hildebrand2002} and \cite{Hildebrand2004}) formed the foundation for understanding V-shaped stationary distributions \cite{Hildebrand2004} and order $n$ convergence when the underlying probabilities are log-concave. \cite{Hildebrand2002} While Hildebrand \cite{Hildebrand2002} posed the question of whether order $n$ convergence occurs if the underlying probabilities are unimodal, an answer to the question was not provided until Lange \cite{Lange2025a}, and it was only answered in the three simplest cases - the simple case, the symmetric function of $n$ case, and the asymmetric function of $n$ case. The results proved in this paper further showed that convergence to stationarity occurs in order $n$ steps for any set of unimodal underlying probabilities. In addition, comparing the result from the general symmetric unimodal case, proved in Lange \cite{Lange2025b}, and the result proved here shows that symmetry is not a critical assumption. \\

    \noindent From here, there are at least three directions for future research. One is to investigate the number of steps needed for convergence to stationarity when the underlying probabilities are bimodal or multimodal. Even though symmetry was shown to not be critical in the analysis of unimodal underlying probabilities, it may prove to be critical when working with bimodal or multimodal underlying probabilities.\\

    \noindent A second future research direction is to investigate convergence to stationarity using other norms, such as the $l^2$ norm. Diaconis, Holmes, and Neal mentioned in their work \cite{DHN2000} that the $l^2$ norm is bounded above by four times the square of the variation distance. One would naturally think that this result is sufficient to show that convergence to stationarity in order $n$ steps occurs using the $l^2$ norm. However, they also proved that convergence to stationarity occurs in order $n \ln n$ steps, not order $n$ steps, when the underlying probabilities are uniform. The issue that Diaconis, Holmes, and Neal mentioned in their paper \cite{DHN2000} is that the chain has a deterministic behavior when it does not flip. A question of interest to ask is whether any similar issues arise for other sets of unimodal underlying probabilities when examining convergence to stationarity using the $l^2$ norm.\\

    \noindent In addition to the two previously stated research questions, we can also pose the question of whether convergence to stationarity exists in order $n$ steps with a single upper bound for every unimodal $\pi$.\\

    \section{Acknowledgments}

    \noindent The authors would like to acknowledge and thank Persi Diaconis, Joshua Isralowitz, Karin Reinhold, Carlos Rodriguez, and Felix Ye for their contributions, suggestions, and encouragement for this project.

\end{document}